\documentclass[a4paper,12pt,reqno]{amsart}
\usepackage{fullpage}
\usepackage{amssymb}
\usepackage{amsthm}
\usepackage{mathtools}
\usepackage{graphicx}
\usepackage{tikz}
\usepackage{tikz-cd}

\usepackage[titletoc]{appendix}

\newcounter{counter2}
\numberwithin{equation}{section}
\numberwithin{counter2}{section}

\newtheorem{thm}[counter2]{Theorem}
\newtheorem{prop}[counter2]{Proposition}
\newtheorem{lemma}[counter2]{Lemma}
\newtheorem{cor}[counter2]{Corollary}
\newtheorem{defn}[counter2]{Definition}
\newtheorem{rem}[counter2]{Remark}

\newcommand{\im}{\operatorname{im}}

\DeclareSymbolFont{script}{U}{eus}{m}{n}
\DeclareMathSymbol{\Wedge}{0}{script}{"5E}

\newcounter{mnotecount}[section]

\renewcommand{\themnotecount}{\thesection.\arabic{mnotecount}}

\newcommand{\mnote}[1]
{\protect{\stepcounter{mnotecount}}$^{\mbox{\footnotesize
$
\bullet$\themnotecount}}$ \marginpar{
\raggedright\tiny\em
$\!\!\!\!\!\!\,\bullet$\themnotecount: #1} }

\def\be{\begin{equation}}

\def\dim{\mbox{dim}}

\def\dim{\mbox{dim}}

\def\ee{\end{equation}}

\def\bea{\begin{eqnarray}}
\def\eea{\end{eqnarray}}

\begin{document}
\title{Joyce structures from quadratic differentials on the sphere}

\author{Timothy Moy}
\address{Department of Applied Mathematics and Theoretical Physics\\ 
University of Cambridge\\ Wilberforce Road, Cambridge CB3 0WA, UK.}
\email{tjahm2@cam.ac.uk} 
\date{\today}

\begin{abstract}
Motivated by known examples of Joyce structures on spaces of meromorphic quadratic differentials,  we consider the isomonodromic deformations of particular second-order linear ODEs with rational potential.  We show the infinitesimal isomonodromic deformations are the kernel of a closed $2$-form arising from the intersection pairing of an algebraic curve defined by the potential.  
This observation enables us to construct Joyce structures on a class of moduli spaces of meromorphic quadratic differentials on the Riemann sphere,  and provides a new,  geometric description of the hyper-K\"ahler structures of previously computed examples.  We focus on the case of moduli of quadratic differentials with poles of odd orders,  where we obtain a complex hyper-K\"ahler metric with homothetic symmetry.   We also include an example corresponding to the moduli space of quadratic differentials with four simple poles,  which is a version of the classical isomonodromy problem that leads to the Painlev\'{e} VI equation.  
\end{abstract}

\maketitle
\section{Introduction}
Joyce structures were proposed by Bridgeland in \cite{B3} as a way of geometrising the study of Donaldson-Thomas invariants.  They are expected to exist on the complex manifold $M$ of stability conditions of a $CY_3$ triangulated category satisfying certain assumptions.  Some properties of these structures were earlier derived in \cite{J} and \cite{JS}.  
 The argument for their existence involves solving Riemann-Hilbert problems with infinite dimensional gauge group,  defined by the variation of the Donaldson-Thomas invariants and as far as the author understands,  solutions to these problems are not yet fully understood.   A geometric axiomatisation of Joyce structures without reference to their origin in Donaldson-Thomas theory is provided in \cite{BS} in terms of a \textit{complex hyper-K\"ahler metric} on the total space $TM$ satisfying certain symmetries compatible with an integral structure on $M$.  \par 
A realisation of spaces of stability conditions was achieved in \cite{BS2},  whereby those of Fukaya categories associated to particular Calabi-Yau $3$-folds fibreing over Riemann surfaces were realised as moduli spaces of quadratic differentials on Riemann surfaces with poles of fixed orders.  There are a number of examples of constructions of the corresponding Joyce structures (called Joyce structures of class $S[A_1]$) starting from the data of a moduli space of quadratic differentials.  In some cases,  the complex hyper-K\"ahler metric can be expressed algebraically in terms of coordinates induced by obvious coordinates on the space $M$ of quadratic differentials.  The case of holomorphic quadratic differentials was handled in \cite{B2}.  A Joyce structure on the moduli space of quadratic differentials with a single pole of order seven on the Riemann sphere was constructed via consideration of the isomonodromy of a family of ODEs with ``deformed cubic oscillator potential" in \cite{BM}.  The case of a single pole of larger odd degree was considered in \cite{DM24} by generalising this construction.  Finally,  the cases of a single pole of order eight and a pair of poles of order three were treated in \cite{BdM}.  
Building upon this earlier work,  to which we are indebted,  here we construct Joyce structures in a more general setting involving multiple poles with odd degree.   \par We emphasise that our aims go further than  extending previous ad-hoc constructions like \cite{DM24} to such a setting.  We seek a geometric description of the hyper-K\"ahler metrics of Joyce structures in this class.  To obtain this,  we first show that there is a geometric characterisation of the isomonodromic deformations of the corresponding ODEs.    
\par
The symplectic geometry and Hamiltonian structure of isomonodromic deformations have long been studied \cite{Bo} and are still of active interest \cite{BK}.   The infinitesimal generators may be obtained as the kernel of a \textit{fundamental 2-form},  and formulae in a general setting are provided in \cite{Y}.  In this work,  within the setting of second order linear ODEs in one dependent variable,  we take a (at least superficially) different approach and construct a fundamental $2$-form in terms of the intersection pairings of a family of algebraic curves.  \par 
We will calculate isomonodromic deformations of a large class of ODEs in the complex variable $x$ of the form
\begin{equation} 
\frac{d^2 Y}{dx^2} = Q(x)Y.  \label{schro}
\end{equation}
Specifically,  we consider those such that the potential takes the form
\begin{align}
Q(x) = \frac{Q_0(x)}{\hbar^2} + \frac{Q_1(x)}{\hbar} + Q_2(x),  
\end{align}
with $Q_0(x),Q_1(x),Q_2(x)$ rational functions depending on a point $\xi \in X$,  where $X$ is a complex manifold,  and $\hbar \in \mathbb{C}^*$ is the \textit{spectral parameter}.  \par  The most important observation will be that, treating $\hbar \in \mathbb{C}^*$ as a constant,  the fundamental $2$-form $\Omega$ arises as a pullback of the family of intersection pairings 
$$
H^1(\Sigma(\xi,\hbar),\mathbb{C}) \times H^1(\Sigma(\xi,\hbar),\mathbb{C}) \to \mathbb{C}
$$
on the algebraic curves $\Sigma(\xi,\hbar)$ defined by $y^2 = Q(x)$.  It is a pullback by a homomorphism,  invariantly defined by the Gauss-Manin connection,  computed at $(\xi,\hbar)$ by the map 
 \begin{align}
 \mu(\xi,\hbar) : T_{\xi}X &\to H^1(\Sigma(\xi,\hbar),\mathbb{C}) \\ 
V &\mapsto [V(y) dx]  \nonumber 
\end{align}
sending vectors to the cohomology class represented by a kind of derivative of the (meromorphic) $1$-form $y dx$ on $\Sigma(\xi,\hbar)$.  \par 
We will identify the parameter $\hbar$ with a coordinate on $\mathbb{CP}^1$ specifying a \textit{twistor fibre}.  Varying $\hbar$,  the family of $2$-forms $\Omega$ will be shown to define an $\mathcal{O}(2)$-valued $2$-form that defines a hyper-K\"ahler metric,  via their usual twistorial characterisation \cite{HKLR} \cite{BE}.  \par
To complete the construction,  we show that there is an immersion from the complex manifold $X$ on which the hyper-K\"ahler metric is defined,  to an open dense subset of $TM$,  with $M$ being the appropriate moduli space of quadratic differentials.  This immersion allows us to push-forward the hyper-K\"ahler structure to $TM$.  We then show the hyper-K\"ahler metric satisfies the definition of a Joyce structure as defined in \cite{BS}.  
\subsection*{Outline} 
The requisite twistor theory of complex hyper-K\"ahler metrics is recalled in \S \ref{HKintro}.  We provide a definition of Joyce structures in \S \ref{Joyceintro} based on that given in \cite{BS}.  In \S \ref{M_defn} we consider the moduli spaces $M$ of quadratic differentials under consideration and construct local coordinates.  In \S \ref{potential_setup} we introduce the complex manifold $X$ of potentials.  For fixed $\hbar$,  the infinitesimal isomonodromic deformations of (\ref{schro}) are shown in \S \ref{curve_approach} and \S \ref{flows_calculation} to be the kernel of the $2$-form $\Omega$.  In \S \ref{Omega_is_O2} we show $\Omega$ defines a complex hyper-K\"ahler structure.  In \S \ref{Joyce_finalisation} it is explained how to transfer the structure from $X$ to the total space $TM$ and that a Joyce structure is obtained.  We include an analysis,  \S \ref{PVI},  of an example (falling outside the family considered in \S \ref{M_defn} - \S \ref{Joyce_finalisation} but handled similarly) relating to the Painlev\'e VI equation,  corresponding to quadratic differentials with four simple poles on the sphere.

\section{Complex hyper-K\"ahler metrics and their twistor theory}\label{HKintro}
In this section we will recall complex hyper-K\"ahler geometry and in particular,  an equivalence between a complex hyper-K\"ahler structure and particular data on \textit{twistor space}.  Good references are \cite{BS}, \cite{BE}.  \par
\begin{defn}[Complex hyper-K\"ahler structure]
Let $X$ be a complex manifold of dimension $4n$.  A \textit{complex} hyper-K\"ahler structure is a holomorphically varying,  non-degenerate bilinear form on $TX$,  the holomorphic tangent bundle of $X$,  together with holomorphic endomorphisms $I,J,K$ of $TX$  satisfying
 \begin{align}
 I^2 &= J^2 = K^2 = IJK = -\operatorname{Id}_{TX}, \\
 I^*g &= J^*g = K^*g = g,  \nonumber \\
 \nabla I &= \nabla J = \nabla K = 0,  \nonumber
 \end{align}
 where $\nabla$ is the Levi-Civita connection. 
\end{defn}
Crucial to constructing our metrics will be families of rank-$2n$ distributions (subbundles of $TX$) depending on a \textit{spectral parameter} $\hbar = u^1/u^0 \in \mathbb{C}^*$ where $[u^0:u^1] \in \mathbb{CP}^1$.  
\begin{defn}[Twistor distribution]\label{twistor_distribution_def}
Let $X$ be a complex manifold of dimension $4n$.  A \textit{twistor distribution} on $X$ is a family $L$ of distributions $L(\hbar) \le TX$ depending on $\hbar \in \mathbb{C}^*$ such that $L(\hbar)$ can be written 
  \begin{align}\label{Lax_distribution}
L(\hbar) = \operatorname{span}\Big\{L_a \coloneqq U_a + \frac{V_a}{\hbar}\Big\}_{a=1}^{2n},  
\end{align}
where $U_a,V_a$ are vector fields on $X$ and 
$TX = \mathrm{span}\{U_{a},  V_{a}\}_{a=1}^{2n}$.  
\end{defn}
 This structure defines a parabolic geometry known under various aliases.  
 For example,  they are called \textit{$(2n,2)$-paraconformal structures} in \cite{BE}.  They are now frequently known as \textit{$(2n,2)$-almost-Grassmannian structures} \cite{CS},  dropping 
 `\textit{almost}' if $L(\hbar)$ is Frobenius integrable for all $\hbar$.  \par 
Importantly,  there is a canonical quaternionic structure $I,J,K$  determined by
\begin{align}\label{quaternions} 
J(U_a) = V_a,  \quad K(U_a) = iV_a
\end{align}
and the quaternion relations.  The span of the $U_a$ and $V_a$ are the $\mp i$ eigenspaces of $I$ respectively. \par 
There is a family of metrics Hermitian for all of the complex structures.  It consists of those that are of the form,  for a choice of basis $L_{a}$,  
\begin{align}
g = \sum_{a=1}^{2n} \sum_{b=1}^{2n} \omega_{ab} (U^{a} \otimes V^{b}   + V^{b} \otimes U^{a}),
\end{align}
where $\{U^{a},V^{a}\}_{a=1}^{2n}$ give the dual trivialisation to $\mathrm{span}\{U_{a},  V_{a}\}_{a=1}^{2n}$ and $\omega_{ab}$ is a skew, non-degenerate matrix of functions.   When $n=1$ this is a conformal class of metrics.  \par 
$L(\hbar)$ Frobenius integrable for all $\hbar$ is equivalent (see Appendix \ref{q_integrability}) to integrability of the $\pm i$-eigenspaces of the complex structures $I,J,K$.  We may think of $L$ as a subbundle of $T(X \times \mathbb{CP}^1)$.  Then,  integrability implies we may construct the space of leaves of this distribution. 
\begin{defn}[Twistor space]
Let $X$ be a complex manifold with a twistor distribution $L$ such that $L(\hbar)$ is Frobenius integrable for each $\hbar$.  Then the \textit{twistor space} associated to $L$ is the space of leaves
\begin{align}
Z = (X \times \mathbb{CP}^1) / L. 
\end{align}
\end{defn}
There are technical complications when it comes to equipping $Z$ with a manifold structure.  We will not need to burden ourselves with these.  We will define geometric structure on $Z$ in this section,  but  we will see the definitions have a clear interpretation in terms of data on $X \times \mathbb{CP}^1$ respecting $L$ in an appropriate way.   \par 
Let $p: Z \to \mathbb{CP}^1$ denote the natural projection.  There are the usual line bundles $\mathcal{O}(n) \to \mathbb{CP}^1$ for $n \in \mathbb{Z}$ with local sections one can think of as homogeneous holomorphic functions of degree $n$ on open subsets of $\mathbb{C}^2 \setminus \{0\}$.  
  \begin{defn}[Relative $\mathcal{O}(2)$-valued 2-form]
Let
 \begin{align}
 \wedge^k_pT^*Z \coloneqq \wedge^k({\operatorname{ker} dp})^*
 \end{align}
denote the bundle of $k$-forms relative to $p$.  
  A relative (to $p$) $\mathcal{O}(2)$-valued 2-form is a section of the bundle $\wedge^2_pT^* Z \otimes p^*\mathcal{O}(2)$. 
  \end{defn}
  We have the relative exterior derivative $d_p: \Gamma(\wedge^2_pT^*Z) \to  \Gamma(\wedge^3_pT^*Z)$ given by 
  \begin{align}
  d_p\omega = d\tilde{\omega}|_{\ker dp},  
  \end{align}
   where $\tilde{\omega}$ is any $2$-form that restricts to $\omega$ on $\operatorname{ker} dp$.  If $d_p\omega = 0$ we say $\omega$ is closed.  $d_p$ extends to local sections of $\wedge^2_pT^* Z \otimes p^*\mathcal{O}(2)$ as we will now explain.  First note that any section of $\wedge^2_p T^*Z \otimes p^*\mathcal{O}(2)$ may be written locally as a product of the form $\omega \otimes p^*\mu$.  We then define
   \begin{align}\label{dp_extension}
   d_p(\omega \otimes p^*\mu) \coloneqq d_p \omega \otimes p^* \mu 
   \end{align}
   and extend by linearity.  
   This is well defined since $d_p$ commutes with multiplication by functions constant in the fibres of $p$.  \par
      The equivalent data on $X \times \mathbb{CP}^1$ in the following proposition will be important for our purposes,  including characterisation of complex hyper-K\"ahler metrics.  The equivalence shown below in the case $Z$ is a well-behaved complex manifold is only for completeness.  
  \begin{prop}\label{practical_definition_twistor_form}
  A closed $\mathcal{O}(2)$-valued relative $2$-form on $Z$ is equivalent to a (usual) $2$-form $\Omega$ on $X \times \mathbb{CP}^1$ annihilating $L$ 
   that may be written
  \begin{align}\label{practical_definition_formula}
\Omega =  \frac{\Omega_-}{\hbar^2} + \frac{i\Omega_{I}}{\hbar} + \Omega_{+},  
  \end{align}
  where $\Omega_-,\Omega_{I},\Omega_{+}$ are (pullbacks of) closed $2$-forms on $X$. 
  \end{prop}
  \begin{proof} 
  Let $q: X \times \mathbb{CP}^1 \to Z$ be the quotient map.  
Suppose we have such an $\Omega$.  Using Cartan's formula,  we have for a vector field $V$ in the distribution $L$,  and $X,Y$ that push forward to vectors in $\ker dp$,  $(\mathcal{L}_V\Omega)(X,Y) = 0$,  so $\Omega$ defines a relative $2$-form $\omega$ on $Z$.  Take the tensor product with the section of $p^*\mathcal{O}(2)$ corresponding to the polynomial $u^1u^1$ to obtain relative $\mathcal{O}(2)$-valued $2$-form.  Now,  for any $2$-form $\tilde{\omega}$ on $Z$ restricting to $\omega$ on $\ker dp$,  it is evident $q^*\tilde{\omega}$ agrees with $\Omega$ up to a $2$-form vanishing when restricted to $X$.  Since $d\Omega$ vanishes when restricted to $X$,   $d\tilde{\omega}|_{\ker dp} = 0$,  and hence the relative $\mathcal{O}(2)$-valued $2$-form is $d_p$-closed.  \par  Conversely,  given a closed $\mathcal{O}(2)$-valued relative $2$-form trivialise the bundle $p^*\mathcal{O}(2)$ by the section corresponding to $u^1u^1$ to obtain a section $\omega$ of $\wedge^2_pT^*Z$.  
We may define a $2$-form $\Omega$ on $X \times \mathbb{CP}^1$ as $\Omega(X,Y) = \omega(dq(X),dq(Y))$ for $X,Y$ tangent to the first factor and insisting $\Omega$ annihilate tangents to the second.  Since the section corresponding to $u^1u^1$ has a zero of order two at $\hbar = 0$,  in order for the tensor product of this section and $\omega$ to be holomorphic,  $\Omega$ necessarily takes the form (\ref{practical_definition_formula}).  $\omega$ is relatively closed so the exterior derivative of $\Omega$ vanishes on restriction to tangents to the first factor,  hence $\Omega_{-},\Omega_{I},\Omega_{+}$ must be all closed.  
These constructions are clearly inverse to one another. 
   \end{proof}
   A closed $\mathcal{O}(2)$-valued $2$-form is of maximal rank if it induces a non-degenerate form on each fibre of the fibration $Z \to \mathbb{CP}^1$.  In terms of $\Omega$ this is equivalent to $\dim \ker \Omega = 2n+1$.  \par
 The twistor characterisation of hyper-K\"ahler metrics can be traced back to \cite{Pe} in four dimensions (although the quaternionic aspect is not mentioned),  \cite{HKLR} 
 in the case of real hyper-K\"ahler metrics in arbitrary dimension.  Another reference in the complex hyper-K\"ahler case is \cite{BE}.  These references discuss construction of a complex manifold $X$ from a complex manifold $Z$ with features making $Z$ the twistor space of some $X$ equipped with a hyper-K\"ahler metric.  Our $X$ is thought of as fixed so the following  is sufficient: 
   \begin{thm}[Twistor characterisation of complex hyper-K\"ahler metrics]\label{twistor_characterisation}
   A complex hyper-K\"ahler structure on a complex manifold $X$ of dimension $4n$ is equivalent to an integrable twistor distribution on $X$ and a closed relative $\mathcal{O}(2)$-valued $2$-form on the twistor space of maximal rank.  
   \end{thm}
     \begin{proof}
  Given a complex hyper-K\"ahler structure on $X$,  we may construct a twistor distribution by choosing a trivialisation $\{U_a\}_{a=1}^{2n}$ for the $i$-eigenspace of $I$,  defining $V_a \coloneqq J(U_a)$ and defining the twistor distribution as per (\ref{Lax_distribution}).  The complex structures (\ref{quaternions}) associated with the twistor distribution agree with the ones coming from the hyper-K\"ahler structure.  $L(\hbar)$ is integrable for all $\hbar$ due to the integrability of the complex structures (which follows from the complex structures being parallel for the Levi-Civita connection).  We now define 
  \begin{align}
  \Omega_- &\coloneqq \frac{1}{2}\big(g(J\cdot,\cdot)-ig(K\cdot,\cdot)\big), \\
  \Omega_I &\coloneqq g(I\cdot,\cdot), \\
  \Omega_+ &\coloneqq  \frac{1}{2}\big(g(J\cdot,\cdot)+ig(K\cdot,\cdot)\big).  
  \end{align}
  With these definitions,  $\Omega$ as in (\ref{practical_definition_formula}) defines a closed relative $\mathcal{O}(2)$-valued $2$-form on the twistor space of maximal rank.  \par 
  Conversely,  given a closed relative $\mathcal{O}(2)$-valued $2$-form on the twistor space of maximal rank,  recall we have the quaternionic structure (\ref{quaternions}) associated to the twistor distribution.  We define the symmetric bilinear form
  \begin{align}
  g = -\Omega_I(I\cdot,\cdot). 
  \end{align}
We consider the equations
  \begin{align}
 \Omega_-(U_a,\cdot) + \hspace{0.4mm} i\Omega_I(V_a,\cdot) &= 0,  \label{L_annihilator_1} \\ 
  \Omega_+(V_a,\cdot) + i\Omega_I(U_a,\cdot) &= 0 \label{L_annihilator_2}
  \end{align}
  that arise from $\Omega$ annihilating $L(\hbar)$ for each $\hbar$.  Since $\dim \ker \Omega = 2n + 1$,  we have $\ker \Omega_- = \text{span}\{V_a\}_{a=1}^{2n}$ and $\ker \Omega_+ = \text{span}\{U_a\}_{a=1}^{2n}$.  This implies  $  \Omega_-(U_a,\cdot)$ and $ \Omega_+(V_a,\cdot)$ give $4n$-linearly independent $1$-forms $a = 1,...,2n$.  Thus (\ref{L_annihilator_1},  \ref{L_annihilator_2}) imply  $\Omega_I$ defines an isomorphism from $TX$ to $T^*X$ and hence that $g$ defines a non-degenerate bilinear form.   Next,  (\ref{L_annihilator_1},  \ref{L_annihilator_2}) together imply 
  \begin{align}
  -(\Omega_- +\Omega_+)(J\cdot,\cdot) = g = -i(\Omega_- -\Omega_+)(K\cdot,\cdot), 
  \end{align}
  from which Hermiticity with respect to $I,J,K$ follows immediately.  Lastly,  the integrability of the complex structures together with closure of the $2$-forms  $\Omega_{-},\Omega_{I},\Omega_{+}$ implies that the complex structures are parallel.  
  \end{proof}
Note that changing the choice of compatible complex structures on the hyper-K\"ahler side of the equivalence corresponds to a M\"obius transformation of $\mathbb{CP}^1$.  \par
\section{Joyce structures}\label{Joyceintro}
Joyce structures were introduced in \cite{B3} as a geometric structure expected to be encoded by the dependence of the Donaldson-Thomas invariants on the choice of stability condition on a $CY_3$ category.  Bridgeland's insight was to interpret the wall-crossing formula as the isomonodromy condition for a connection with infinite dimensional gauge group,  in particular,   the Poisson automorphisms of the algebraic torus.  The most striking feature of the proposed structure is a (meromorphic) complex hyper-K\"ahler metric with a number of symmetries on the total space $TM$,  where $M$ is the complex manifold of stability conditions.  The hyper-K\"ahler metric arises as a consequence of the isomonodromic deformations defining a twistor distribution on $TM$.  The existence of solutions to the Riemann-Hilbert problem of reconstructing the twistor distribution from the wall-crossing data is not known in general for the class of $CY_3$ categories conjectured to carry Joyce structures on their spaces of stability conditions.  \par 
A geometric axiomatisation of Joyce structures is provided in \cite{BS}.  Let $(M,\omega)$ be a complex manifold of complex dimension $2n$ with holomorphic symplectic\footnote{To avoid confusion,  in \cite{BS},  Joyce structures compatible with a non-degenerate symplectic form as we will assume here are called \textit{strong} Joyce structures. } form $\omega$.  Let $TM$ be the holomorphic tangent bundle equipped with the natural projection $\pi: TM \to M$.  A lattice $\mathcal{G}_p$ at $p$ is a discrete subgroup of $T_pM$ such that complex multiplication $\mathcal{G}_p \otimes_{\mathbb{Z}} \mathbb{C} \to TM$ is an isomorphism.  A holomorphic bundle of lattices $\mathcal{G} \hookrightarrow TM$ defines a flat linear connection $\nabla$ on $TM$ by distinguishing the $2n$-dimensional vector space of parallel sections.  Specifically,  the vector space of parallel sections is generated by those sections taking values in the lattice.  \par The first ingredient of a Joyce structure on $M$ is the following.  
\begin{defn}[Period structure with symplectic form]
Let $M$ be a complex manifold of dimension $2n$.  A period structure on $M$ is a holomorphic bundle of lattices $\mathcal{G} \to TM$ together with a holomorphic symplectic form $\omega$ on $M$ that takes rational values on the lattice,  and a holomorphic vector field $E_0$ such that $\nabla E_0 = \operatorname{Id}_{TM}$ where $\nabla$ is the flat connection associated to the lattice.  
\end{defn}
 Next,  the \textit{non-linear data} that completes a Joyce structure is a complex hyper-K\"ahler structure on the complex manifold $TM$,  along with some symmetry conditions.  \par First,  we recall the geometry of the tangent bundle.  There is a short exact sequence 
 \begin{center}
\begin{tikzcd}
\ker d\pi \arrow[r, hook] & T(TM) \arrow[r, "\pi_*", two heads] & \pi^*TM \arrow[ll, "\tau" description, bend left=20].  
\end{tikzcd}
\end{center}
Here $\pi_*$ is the obvious projection and $\tau$ is the vertical lift of a tangent vector.  
This yields a canonical linear endomorphism $\nu: T(TM) \to T(TM)$ satisfying $\nu \circ \nu = 0$ called a \textit{null structure} (see \cite{D}) given by $\nu \coloneqq \tau \circ \pi_*$.  
\begin{defn}[Joyce structure]\label{J}
Let $M$ be a complex manifold of dimension $2n$.  Let $\pi: TM \to M$ be the natural projection.  A (strong) Joyce structure is a period structure with symplectic form $\omega$ on $M$ together with a complex hyper-K\"ahler structure $(g,I,J,K)$ on the total space $TM$ satisfying the following:
\begin{enumerate}
\item $\pi^*\omega = \Omega_{-}$,  where $\Omega_{-} \coloneqq (g(J\cdot,\cdot) - ig(K\cdot,\cdot))/2$.
\item $\nu = (J-iK)/2$.
\item $d\pi \circ iI = d\pi$.
\item $\iota^*I = I,  \iota^*(J \pm iK) = -(J \pm iK),  \iota^*g = -g$,  where $\iota$ is the linear involution in the fibres given by multiplication by $-1$.  
\item $\mathcal{L}_EI = 0,  \mathcal{L}_{E}(J \pm iK) = \mp(J \pm iK),  \mathcal{L}_Eg = g$,  where $E$ is the horizontal lift of $E_0$ with respect to  $\nabla^{GM}$.  
\item The complex hyper-K\"ahler structure is invariant under translations by the lattice.
\end{enumerate}
\end{defn}
The first three conditions are equivalent to the hyper-K\"ahler structure on $TM$ defining a \textit{normalised affine symplectic fibration} in the sense of \cite{BS}.  

\section{Quadratic differentials and the moduli space $M = {\rm Quad}(\bar{m})$}\label{M_defn}
A meromorphic quadratic differential on $\mathbb{CP}^1$ is a meromorphic section of $K_{\mathbb{CP}^1}^2$.  With respect to the affine coordinate $x$ for $\mathbb{CP}^1$ we may write any such quadratic differential as $\phi = \mathcal{Q}(x) dx^2$ for a rational function $\mathcal{Q}(x)$ on $\mathbb{C}$.  \par 
We will construct Joyce structures on open subsets of the moduli space $M = \operatorname{Quad}(\bar{m})$ consisting of meromorphic quadratic differentials on $\mathbb{CP}^1$ with simple zeroes and with exactly $N+1$ poles of fixed orders given by an unordered tuple of integers $\bar{m}$.  We will assume that:
\begin{itemize}
\item $N \ge 0$ so that there is at least one pole.  
\item $\bar{m}$ consists of odd integers.
\item There is at least one element of $\bar{m}$ greater than or equal to five.  
\item $\bar{m} \ne \{5\}$. 
\end{itemize} 
Given the conditions above,  we are assured that $M$ is non-empty of dimension greater than or equal to two.   
 The condition that there is a pole of order greater than or equal to five is mainly for convenience in choosing local coordinates.  Indeed,  later we will handle separately the case with four simple poles,  along the same lines as the restricted class above.  Note that the moduli spaces considered here include the cases of $\bar{m} = \{2n + 5\}$ for $n \ge 1$ which were considered in \cite{DM24}. \par
We identify quadratic differentials $\phi_1$ and $\phi_2$ if there is an automorphism $f$ of $\mathbb{CP}^1$ with $f^*\phi_2 = \phi_1$.  Acting by a M\"obius transformation we may restrict ourselves to consider representative quadratic differentials with a pole at $x = \infty$ on $\mathbb{CP}^1$ taking the pole at infinity to be of order at least five.  To such a quadratic differential we may associate a polynomial: If $\phi = \mathcal{Q}(x)dx^2$ take $\mathcal{Q}(x)$ and subtract the principal parts of its Laurent series at poles in $\mathbb{C}$.  We may act by a M\"obius transform fixing $\infty$ so that the sum of the zeroes of this polynomial vanishes.  Next,  we can act by a dilation to set the coefficient of the highest power of $x$ to $1$.  All this implies $\phi = Q_0(x) dx^2$ where
\begin{align}\label{coordinate_rep}
Q_0(x) &= x^{2m_\infty-5} + \sum_{i=0}^{2m_\infty-7} a^{(\infty)}_ix^i + \sum_{\alpha=1}^{N}\sum_{i=1}^{2m_\alpha-1}\frac{a_i^{(\alpha)}}{(x-w_\alpha)^i},
\end{align}
where $m_\infty,m_1,...,m_N$ are positive integers determined by the pole orders,  the $w_\alpha$ give the locations of the poles in $\mathbb{C}$ and the $a_i^{(\alpha)}$ give Laurent coefficients.  
 Such a representative may not be unique.  
 What is true,  following from elementary properties of M\"obius maps,  is that there are finitely many representatives of the above form.  This implies we may then take the parameters $\{a_i^{(\alpha)}\hspace{-1mm}, w_\alpha\}$ to be \textit{local} coordinates on $M$.  Recall that we would like $N+1$ distinct poles (including infinity) of fixed orders and simple zeroes.  This condition can be expressed as the non-vanishing on $M$ of some holomorphic function of the parameters.  Some obvious consequences of this condition are that $a^{(\alpha)}_{2m_\alpha-1} \ne 0$ for each $\alpha$ and the $w_\alpha$ must be distinct.  The coordinates distinguish an ordering on the poles with orders
 $$
 \bar{m} = \{2m_\infty-1,  2m_1-1, ...,2m_N-1\},  
 $$
 where $m_\infty \ge 3$ and $m_\alpha \ge 1$ for $\alpha = 1,...,N$.  
A simple count gives 
\begin{align}
\operatorname{dim} M = 2m_\infty - 6 + \sum_{\alpha=1}^{N}2m_\alpha.  
\end{align}
Associated to a quadratic differential is a smooth algebraic curve called the \textit{spectral curve}.  
It is a branched cover of $\mathbb{CP}^1$.  It may be realised as the submanifold
\begin{align}
\Sigma_0 \coloneqq \{ l \in  K_{\mathbb{CP}^1}(D) \ | \ l \otimes l = \phi\} 
\end{align}
 of $K_{\mathbb{CP}^1}(D)$ where $D$ is the divisor on $\mathbb{CP}^1$ given by 
\begin{align}
D = m_\infty \cdot \infty + \sum_{\alpha=1}^{N} m_\alpha \cdot w_\alpha
\end{align}
and thinking of $\phi$ as a holomorphic section of $K_{\mathbb{CP}^1}(D) \otimes K_{\mathbb{CP}^1}(D)$. 
More practically,  the \textit{affine piece} is realised as a subspace of $\mathbb{C}^2$ by the equation $y_0^2 = Q_0(x)$ away from the poles.  The covering map $\Pi: \Sigma_0 \to \mathbb{CP}^1$ is branched precisely at the poles and zeroes of $\phi$.  There is a \textit{tautological} $1$-\textit{form} $\psi$,  a meromorphic section of the canonical bundle $K_\Sigma$ defined (up to a choice of sign) by the condition $\Pi^*\phi = \psi \otimes \psi$.  It has local representation $\psi = y_0 \,  dx$ in the charts for $\Sigma$ induced by the chart for $\mathbb{CP}^1$ in which we may write $\phi = Q_0(x) dx^2$.  $\psi$ has poles at the pre-images under $\Pi$ of poles of $\phi$ of order greater than one.  Precisely: $\psi$ has a pole of order $2m_\alpha-2$ at the branch point $w_\alpha$ (note that if $\phi$ has a simple pole at $w_\alpha$ then $\psi$ will be holomorphic at the corresponding branch point).  The $1$-form has no residues at its poles since pulling back by the covering involution $y \mapsto -y$ sends $\psi \mapsto -\psi$.  An important consequence of the vanishing residues is that $\psi$ defines an element $[\psi] \in H^1(\Sigma_0,\mathbb{C})$. \par 
An application of the Riemann-Hurwitz formula shows that for some $p \in M$,  the spectral curve $\Sigma_0(p)$ defined by the equation \begin{align}
y_0^2 = Q_0(x)
\end{align} 
has genus
\begin{align}\label{ngendef}
n \coloneqq \frac{\operatorname{dim} M}{2}.  
\end{align}
There is a holomorphic rank-$2n$ vector bundle $\mathcal{E} \to M$ with fibres
\begin{align}
\mathcal{E}_p \coloneqq  H^1(\Sigma_0(p),\mathbb{C}). 
\end{align}
This is equipped with a canonical linear connection,  the \textit{Gauss-Manin} connection $\nabla^{GM}$ defined as the flat connection with parallel sections taking values in the dual lattice to the lattice defined by a canonical basis of cycles $\{\gamma_{a}\}_{a=1}^{2n}$.  There is also a canonical section $Z$ taking the value $[\psi]$ at a point $p \in M$.  \par 
Define a vector bundle homomorphism $\mu_0: TM \to \mathcal{E}$ by
$$
\mu_0(U) = \nabla^{GM}_U Z.  
$$
Let $\{\nu_{a}\}_{a=1}^{2n}$ denote the dual basis to $\{\gamma_{a}\}_{a=1}^{2n}$.  The corresponding sections of $\mathcal{E}$ are parallel for the Gauss-Manin connection.  The $\nu_{a}$-coefficient of an arbitrary class may be evaluated by integration around the cycle $\gamma_{a}$.  

The explicit realisation of the curves as branched covers allows us to find a representative meromorphic $1$-form for $\mu_0(U)$:
\begin{align}
\mu_0(U) = \nabla^{GM}_U Z = \sum_{a=1}^{2n} U\bigg( \int_{\gamma_a} \psi \bigg) \nu_a = \sum_{a=1}^{2n} \bigg( \int_{\gamma_a} U(y_0)dx \bigg) \ \nu_a = [U(y_0)dx]. 
\end{align}
Here $U(y_0)$ is defined as the derivative of the function $y_0$ (thought of as a function of $x$ and the coordinates on $M$) in the direction $U$.  
For brevity we will adopt the notation:
\begin{align}
U(\psi) \coloneqq U(y_0)dx. 
\end{align}
so that $\mu_0(U) = [U(\psi)]$.  \par
We will see that $\mu_0$ is an isomorphism and use it to equip $M$ with a period structure.  \par 
In each fibre we have the usual intersection pairing $\langle \cdot | \cdot \rangle_0 : H^1(\Sigma_0,\mathbb{C}) \times H^1(\Sigma_0,\mathbb{C}) \to \mathbb{C}$,  
\begin{align}
\langle [\alpha] | [\beta] \rangle_0 \coloneqq \int_{\Sigma_0} \alpha \wedge \beta,  
\end{align}
where $\alpha, \beta$ are smooth representatives.  Since $\{\gamma_{a}\}_{a=1}^{2n}$ is a canonical basis,  $\langle \nu_{a} | \nu_{b} \rangle_0 = \omega_{ab}$ where $\omega_{ab}$ is the standard symplectic matrix of rank $2n$. \par 
These intersection pairings define a non-degenerate holomorphic skew-form (which we will also denote by $\langle \cdot | \cdot \rangle_0$) on $\mathcal{E}$.  Define\footnote{This definition of $\omega$ differs by a constant factor from the definition in \cite{DM24}.}
\begin{align}
\omega \coloneqq \mu_0^*\langle \cdot | \cdot \rangle_0,  
\end{align}
which is a $2$-form on $M$.  
\begin{prop}[Residue formula]\label{residue_formula_prop} Let $\mathcal{P}_0(p) \subseteq \Sigma_0(p)$ denote the set of poles of $\psi$.  The $2$-form
$\omega$ is closed and given by 
\begin{align}\label{intersection_residue}
\omega(U,V) = 2\pi i \hspace{-2mm} \sum_{x \in \mathcal{P}_0(p)} \,  \underset{x}{\operatorname{Res}}( d^{-1} U(\psi)V(\psi)),  
\end{align}
where $U,V \in T_pM$ for some $p \in M$ and the anti-derivative $d^{-1}U(\psi)$ is a meromorphic function defined on open discs about each $x \in \mathcal{P}_0(p)$ and chosen arbitrarily. 
\end{prop}

\begin{proof} 
Take disjoint open discs $D_\alpha$ containing $x = w_\alpha$ and $D_\infty$ containing $\infty$  and let $D = \cup_{\alpha=1}^N D_\alpha \cup D_\infty$.  Then take open discs $S_{\alpha}$ to have, say,  half the radii and similarly for $S_\infty$ and let $S = \cup_{\alpha=1}^N S_\alpha \cup S_\infty$. 
Take a smooth function $\rho$ which takes the constant value $1$ on the $S$ and has support in $D$.  A meromorphic anti-derivative $d^{-1}U(\psi)$ defined on $D$ exists since $\psi$ and hence $U(\psi)$ is meromorphic with no residues.  
Then $U(\psi)- d(\rho d^{-1}U(\psi))$ is a globally defined smooth $1$-form that is holomorphic on the complement of $D$ and
\begin{equation}
\mu_0(U) = [U(\psi)] = [U(\psi)- d(\rho d^{-1}U(\psi))].  
\end{equation}
Since the wedge product of holomorphic $1$-forms vanishes,  
\begin{align}
\omega(U,V) &= \int_{D} \Big( U(\psi) - d(\rho d^{-1}U(\psi)) \Big) \wedge \Big( V(\psi) - d(\rho d^{-1}V(\psi)) \Big) \nonumber \\
 &= \int_{D} d\Big(\rho d^{-1}U(\psi) d(\rho d^{-1}V(\psi)) + (\rho d^{-1}V(\psi))  U(\psi) - (\rho d^{-1}U(\psi))  V(\psi) \Big) \nonumber \\ 
& = \oint_{\partial S} d^{-1}U(\psi)V(\psi),  
\end{align} 
where we have used Stokes' theorem and that $\rho = 1$ on $\partial S$ in the last step.  Here the integration over $\partial S$ is taken with respect to the usual orientation.  
By Cauchy's integral formula,  
\begin{align}
\omega(U,V) = \int_{\partial S} d^{-1}U(\psi)V(\psi)  = 2\pi i \hspace{-2mm} \sum_{x \in \mathcal{P}_0(p)}  \underset{x}{\operatorname{Res}}(d^{-1}U(\psi)  V(\psi)).  
\end{align} 
Finally,  differentiating inside the contour integral and integration by parts yields
\begin{align}
U(\omega(V,W)) = \  \nonumber &V(\omega(U,W)) + W(\omega(V,U)) \\ &+ \omega([U,V],W) - \omega([U,W],V) - \omega([W,V],U),  
\end{align}
so that $d\omega = 0$.  
\end{proof}
Using this residue formula we can prove the following: 
\begin{prop}
 $\omega$ is a symplectic form.  
\end{prop}
 \begin{proof} 
 It remains to show $\omega$ is non-degenerate.  It is convenient to write $a_{2m_\alpha}^{(\alpha)} \coloneqq w_\alpha$ for each $\alpha = 1,...,N$.  First,  we will show that $\omega$ has a block diagonal decomposition in the coordinate trivialisation
\begin{align}
\bigg\{ \frac{\partial}{\partial a_{i}^{(\alpha)}} \bigg\}_{\alpha = 1,...,N}^{i = 1,...,2m_\alpha} \cup \bigg\{ \frac{\partial}{\partial a_{i}^{(\infty)}} \bigg\}^{i = 0,...,2m_\infty - 7}
\end{align}
for $TM$.  
Then we will show each block is triangular with non-vanishing diagonal entries.  
 At $x = w_\alpha$ we have
\begin{align}\label{order_diff_1}
\frac{\partial y_0}{\partial a_i^{(\alpha)}} = O((x-w_\alpha)^{m_\alpha - \frac{1}{2} - i}),  \quad 
\frac{\partial y_0}{\partial a_j^{(\beta)}} = O((x-w_\alpha)^{m_\alpha - \frac{1}{2}}) 
\end{align}
for $\alpha \ne \beta$ and 
\begin{align}\label{cohomology_coordinate_diff}
\mu_0\bigg(\frac{\partial}{\partial a_i^{(\alpha)}}  \bigg) = \Bigg[ \frac{\partial y_0}{\partial a_i^{(\alpha)}} dx \Bigg] .  
\end{align}
The representative meromorphic $1$-form in (\ref{cohomology_coordinate_diff}) has a pole only at $x = w_\alpha$.  Let $s_\alpha$ be a local coordinate on $\Sigma_0(p)$ at the branch point $x = w_\alpha$ so that $s_\alpha^2 = (x-w_\alpha)$ and $dx = 2s_\alpha ds_\alpha$.  Then 
\begin{align}
\omega \bigg(\frac{\partial }{\partial a_i^{(\alpha)}},  \frac{\partial }{\partial a_j^{(\beta)}} \bigg) = \,  &\underset{x = w_\alpha}{\operatorname{Res}} \bigg( d^{-1} \bigg( \frac{\partial y_0}{\partial a_i^{(\alpha)}}  dx \bigg) \frac{\partial y_0}{\partial a_j^{(\beta)}} dx \bigg) \nonumber +  \,  \underset{x = w_\beta}{\operatorname{Res}} \bigg( d^{-1} \bigg( \frac{\partial y_0}{\partial a_i^{(\alpha)}}  dx \bigg)  \frac{\partial y_0}{\partial a_j^{(\beta)}} dx \bigg) \nonumber \\
= \,  & \underset{s_\alpha = 0}{\operatorname{Res}}(O(s_\alpha^{4m_\alpha+1-2i}) ds_\alpha) +  \underset{s_\beta = 0}{\operatorname{Res}}(O(s_\beta^{4m_\beta+1-2j}) ds_\beta) = 0, 
\end{align}
as $i \le 2m_\alpha$ and  $j \le 2m_\beta$.  Similarly,  at the pole at infinity
\begin{align}\label{order_diff_inf_1}
\frac{\partial y_0}{\partial a_i^{(\alpha)}} = O(x^{-m_\infty+\frac{5}{2}-i}),  \quad \frac{\partial y_0}{\partial a_j^{(\infty)}} = O((x^{j-m_\infty+\frac{5}{2}})).
\end{align}
While,  at a pole $x = w_\alpha$ in $\mathbb{C}$ 
\begin{align}\label{order_diff_inf_2}
\frac{\partial y_0}{\partial a_j^{(\infty)}} = O((x-w_\alpha)^{m_\alpha - \frac{1}{2}}),
\end{align}
so that taking a coordinate $s_\infty^{-2} = x$ about the branch point $x = \infty$ we have 
\begin{align}
\omega\bigg(\frac{\partial }{\partial a_i^{(\alpha)}},  \frac{\partial }{\partial a_j^{(\infty)}} \bigg) = \underset{s_\alpha=0}{\operatorname{Res}}(O(s_\alpha^{4m_\alpha+1-2i})ds_\alpha) + \underset{s_\infty = 0}{\operatorname{Res}}(O(s_\infty^{4m_\infty - 15 - 2j + 2i})ds_\infty)=0,
\end{align}
taking $\alpha = 1,...,N,  i = 1,...,2m_\alpha$ and $j = 0,..., 2m_\infty - 7$.  This gives us a block diagonal decomposition with blocks labelled by $\alpha = 1,...,N$ and $\infty$.  Each block is given by $\omega$ restricted to the vector fields corresponding to the Laurent coefficients at each pole.  For a finite pole $x = w_\alpha$ we have
\begin{align}\label{triangular}
\omega\bigg(\frac{\partial }{\partial a_i^{(\alpha)}},  \frac{\partial }{\partial a_j^{(\alpha)}} \bigg) \nonumber = \,  &\underset{x = w_\alpha}{\operatorname{Res}} \bigg(  d^{-1} \bigg( \frac{\partial y_0}{\partial a_i^{(\alpha)}}  dx \bigg)  \frac{\partial y_0}{\partial a_j^{(\alpha)}} dx \bigg) \\ = \,  &\underset{s_\alpha = 0}{\operatorname{Res}}\bigg( \frac{s_\alpha^{4m_\alpha+1-2i-2j}}{a^{(\alpha)}_{2m_\alpha-1}(2m_\alpha+1-2i)}ds_\alpha+ O(s_\alpha^{4m_\alpha+2-2i-2j})ds_\alpha\bigg),  
\end{align}
which vanishes for $i + j \le 2m_\alpha$.  Since $a_{2m_\alpha-1}^{(\alpha)} \ne 0$ on $M$ (due to the poles of $\phi$ having prescribed orders),  we see that (\ref{triangular}) is non-zero for $i + j = 2m_\alpha + 1$ which shows that the $2m_\alpha \times 2m_\alpha$ matrix with entries (\ref{triangular}) for $i,j = 1,...,2m_\alpha$ is triangular with non-zero diagonal entries (it is triangular in the sense opposite to the usual one). 
Similarly,  for the block corresponding to the pole at infinity,  
\begin{align}\label{triangular_inf}
\omega\bigg(\frac{\partial }{\partial a_i^{(\infty)}},  \frac{\partial }{\partial a_j^{(\infty)}} \bigg) = \,  &\underset{x = \infty}{\operatorname{Res}} \bigg( d^{-1} \bigg( \frac{\partial y_0}{\partial a_i^{(\infty)}}  dx \bigg)  \frac{\partial y_0}{\partial a_j^{(\infty)}} dx \bigg) 
\end{align}
vanishes for $i + j \le 2m_\infty-8$ and is non-zero for $i + j = 2m_\infty - 7$,  so that the last block is triangular.  
Accordingly,  $\omega$ is non-degenerate.  
\end{proof}
\begin{cor}
$\mu_0: TM \to \mathcal{E}$ is an isomorphism.  
\end{cor}
\begin{proof}
$\mu_0$ must be injective since $\omega = \mu_0^*\langle \cdot | \cdot \rangle_0$ is non-degenerate.  $TM$ and $\mathcal{E}$ are bundles of rank $2n$,  so $\mu_0$ is an isomorphism.  
\end{proof}
We may pullback the lattices given by the integral span of $\{\nu_{a}\}_{a=1}^{2n}$ by the isomorphism $\mu_0$ to obtain a bundle of lattices $\mathcal{G} \to TM$.  \par Consider the functions (periods)
\begin{align}
z^{a} = \oint_{\gamma_{a}} \psi.  
\end{align}
\begin{prop}
The functions $\{z^{a}\}_{a=1}^{2n}$ give local coordinates on $M$.  
\end{prop}
\begin{proof}
As before it is convenient to take $a_{2m_\alpha}^{(\alpha)} \coloneqq w_\alpha$.  From Riemann's bilinear relations
\begin{align}
\Bigg\langle \frac{\partial \psi}{\partial a_{i}^{(\alpha)}} \Bigg| \frac{\partial \psi}{\partial a_{j}^{(\beta)}}   \Bigg\rangle &= \sum_{a=1}^{n} \underset{\gamma_{a}}{\oint}\frac{\partial \psi}{\partial a_{i}^{(\alpha)}}  \underset{\gamma_{a+n}}\oint\frac{\partial \psi}{\partial a_{j}^{(\beta)}} - \underset{\gamma_{a+n}}\oint \frac{\partial \psi}{\partial a_{i}^{(\alpha)}}  \underset{\gamma_{a}}\oint \frac{\partial \psi}{\partial a_{j}^{(\beta)}} \nonumber \\&= \sum_{a=1}^{n} \frac{\partial z^{a}}{\partial a_{i}^{(\alpha)}}\frac{\partial z^{a+n}}{\partial a_{j}^{(\beta)}} - \frac{\partial z^{a+n}}{\partial a_{i}^{(\alpha)}}\frac{\partial z^{a}}{\partial a_{j}^{(\beta)}} = \sum_{a=1}^{2n}\sum_{b=1}^{2n} \frac{\partial z^{a}}{\partial a_{i}^{(\alpha)}} \omega_{ab}\frac{\partial z^{b}}{\partial a_{j}^{(\beta)}},   \label{Riemann_bilinear_periods}
\end{align}
where $\omega_{ab}$ is the standard symplectic matrix of rank $2n$.  
The matrix with entries given by the left-hand side is the (non-degenerate) matrix of $\omega$ in the coordinate basis for the coordinates $\{a_{i}^{(\alpha)}\hspace{-1mm},w_\alpha\}$ which implies the Jacobian matrix must be non-degenerate.  
\end{proof}
In fact,  the equation (\ref{Riemann_bilinear_periods}) is the statement that 
\begin{align}
\omega =  \frac{1}{2}\sum_{a=1}^{2n} \sum_{b=1}^{2n} \omega_{ab}dz^{a} \wedge dz^{b}.  
\end{align}
The lattice $\mathcal{G} \to TM$ is given by the span of vector fields with integer coefficients with respect to the $\{z^{a}\}_{a=1}^{2n}$ coordinate basis.  The coordinates $\{z^{a}\}_{a=1}^{2n}$ are therefore flat for the associated connection $\nabla$.  The vector field $E_0$ defined by
\begin{align}
E_0 = \sum_{a=1}^{2n} z^{a}\frac{\partial}{\partial z^{a}} 
\end{align}
satisfies
\begin{align}
\nabla E_0 = \operatorname{Id}_{TM}.  
\end{align}
Altogether we have shown
\begin{thm}
The lattice $\mathcal{G}\to TM$,  together with $\omega$ and the vector field $E_0$ define a period structure with symplectic form on $M$.  
\end{thm}
\begin{rem}
The above construction differs from the conjectured general construction of Joyce structures of class $S[A_1]$ outlined in \cite{BdM},  which uses the natural period structure on $M$ coming from Donaldson-Thomas theory.  The difference comes from considering integral homology of the curve $\Sigma_0(p)$ punctured at the set of poles $\mathcal{P}_0(p)$ of $\psi$. Our lattice is dual to the integral cycles $\{\gamma_{a}\}$ defining a basis for $H_1(\Sigma_0(p),\mathbb{Z})$ whereas the integral lattice in the proposed construction is dual to a basis for $H_1^{-}(\Sigma_0(p) \setminus \mathcal{P}_0(p),\mathbb{Z})$,  which is the subgroup of $H_1(\Sigma_0(p) \setminus \mathcal{P}_0(p),\mathbb{Z})$ consisting of classes anti-invariant under the covering involution.  The natural homomorphism $H_1^{-}(\Sigma_0(p) \setminus \mathcal{P}_0(p),\mathbb{Z}) \to H_1(\Sigma_0(p),\mathbb{Z})$ induced by the inclusion $\Sigma_0(p) \setminus \mathcal{P}_0(p) \hookrightarrow \Sigma_0(p)$ is injective and of full-rank but not in general surjective.  \par 
The resulting Joyce structures will share the skew-form $\omega$ on $M$,  underlying hyper-K\"ahler metric and Euler vector field,  but the period coordinates will be different,  related by a constant linear transformation.  
\end{rem}
\section{Potentials and the complex manifold $X$}\label{potential_setup}
Having introduced the geometry on the $2n$-dimensional moduli space $M$ of quadratic differentials,  we turn our attention to a bundle $X \to M$ with half-dimensional fibres,  the points of which will,  together with a spectral parameter $\hbar \in \mathbb{C}^*$ determine an ODE.  \par 
The potentials considered here may be expanded in orders of $\hbar$ as follows: 
\begin{align}\label{hbar_expansion_potential}
Q(x) = \frac{Q_0(x)}{\hbar^2}+\frac{Q_1(x)}{\hbar}+Q_2(x). 
\end{align}
The leading order term in $\hbar$ is determined by the quadratic differential $\phi = Q_0(x)dx^2$,  that is,  a point $p \in M$.  \par 
Let $\mathcal{B}(p)$ denote the set of branch points in $\mathbb{C}$ of the curve $\Sigma_0(p)$.  Specifically,  $\mathcal{B}(p)$ consists of the $N$ poles and the zeroes of $\phi$ in $\mathbb{C}$.   
We introduce the complex manifold $X$ that parametrises the choice of point $p \in M$ together with a \textit{deformation term}
\begin{align}
Q_1(x) = \sum_{I=1}^{n} \frac{p_I}{x-q_I} + R(x),  
\end{align}
where the 
$q_{I} \in \mathbb{C} \setminus \mathcal{B}(p)$ are distinct,  $p_I$ is fixed up to a sign by $p_I^2 = Q_0(q_I)$ and where $R(x)$ is the general rational function so that the fibre $X_p$ may be thought of,  via defining 
\begin{align}\label{sigma}
\sigma \coloneqq \frac{Q_1(x)}{2y_0}dx,  
\end{align}
as the space of meromorphic $1$-forms $\sigma$ on $\Sigma_0(p)$ with simple poles with residues $\pm 1/2$ at each pair of points corresponding to $x = q_I$ and no other poles.  This means,  $R(x)$ must be the general rational function with a pole of order at most $m_\infty-4$ at infinity and poles of order at most $m_\alpha$ at the points $w_\alpha$.  

It turns out it is convenient to use parameters $\{v_I\}$ for the choice of $R(x)$ so that 
\begin{align}\label{Rfrac}
R(x) = \frac{\displaystyle{\sum_{I=1}^{n}F(q_I)\Big(2p_Iv_I - \sum_{K \ne I}^{n}\frac{p_K}{q_I-q_K}\Big)\prod_{J \ne I}^n \frac{(x-q_J)}{(q_I-q_J)}}}{F(x)},
\end{align}
where 
\begin{align}
F(x) = \prod_{\alpha=1}^{N}(x-w_\alpha)^{m_\alpha}.  
\end{align}

To see that (\ref{Rfrac}),  for fixed $q_I$,  gives the general such rational function note that,  multiplying the formula (\ref{Rfrac}) through by $F(x)$ yields the general polynomial of degree at most $n-1$,  since it is specified by its values at $n$ points $q_I$ which can be freely chosen. \par 
Since $\{a_{i}^{(\alpha)}\hspace{-1mm},w_\alpha,  q_I,  v_I\}$ give local coordinates,  we see that $X$ is $4n$-dimensional.  \par 

The remaining term in the potential (\ref{hbar_expansion_potential}) is completely determined by the point in $X$ if we impose two conditions on $Q(x)$.  The first is that $Q(x)$ has a particular form of an \textit{apparent singularity} at $x = q_I$.  It is that $Q(x)$ has Laurent expansions
\begin{align}\label{apparent}
Q(x) = \frac{3}{4(x-q_I)^2} + \frac{u_I}{x-q_I} + u_I^2 + O(x-q_I)
\end{align}
for some functions $u_I$ on $X$.  The second condition is that $Q_2(x)$,  as with $Q_1(x)$,  has poles at the poles of $Q_0(x)$ in such a way that $Q_2(x)/2y_0 \,  dx$ has poles only at  the points corresponding to $x = q_I$.  \par 
That the existence of apparent singularities is necessary for the existence of isomonodromic deformations of scalar ODEs was shown by Poincar\'e \cite{Po}. The monodromy of the equation (\ref{schro}) is fixed at these points in such a way that analytic continuation around the singularity changes the sign of the solution.  \par 
The reason for the choice of parameters $\{v_I\}$ is that 
\begin{align}
Q_1(x) := \frac{p_I}{x-q_I} + 2p_Iv_I + O(x-q_I).  
\end{align}
Then,  since $Q_0(x) = p_I^2 + O(x-q_I)$,  the condition (\ref{apparent}) implies $u_I = p_I/\hbar + v_I$ and the two conditions together force
\begin{align}
Q_2(x) = \sum_{I=1}^n \frac{3}{4(x-q_I)^2} +  \sum_{I=1}^n \frac{v_I}{x-q_I} + S(x),  
\end{align}
where
\begin{align}
S(x) = \frac{\displaystyle{{\sum_{I=1}^{n} F(q_I) \Big(v_I^2 - \sum_{K \ne I}^n\frac{3}{4(q_I-q_K)^2} - \sum_{K \ne I}^n \frac{v_K}{q_I-q_K}\Big)\prod_{J \ne I}^n \frac{(x-q_J)}{(q_I-q_J)}}}}{F(x)}.  
\end{align}
We consider the equation $y^2 = Q(x)$ determined by a $\xi \in X$ and $\hbar \in \mathbb{C}^*$.  In Appendix \ref{generic_non_singularity},  we prove a fact that we will need later,  that for $(\xi,\hbar) \in X \times \mathbb{C}^{*}$ in the complement of the vanishing set of a non-zero holomorphic function $\Delta: X \times \mathbb{C}^* \to \mathbb{C}$,  $Q(x)$ has simple zeroes.  The Riemann-Hurwitz formula implies that,  given $Q(x)$ has simple zeroes,  the genus of the smooth algebraic curve $\Sigma(\xi,\hbar)$ defined by $y^2 = Q(x)$ is $2n$.  
\begin{rem}
One may note that there are redundant parameters in the sense each choice of potential $Q(x)$ corresponds to an $n$-dimensional submanifold of the complex manifold $X$.  The application to Joyce structures motivates preserving these extra parameters.  
\end{rem}

\begin{rem}
Letting $\mathcal{S}_0(p) \subset \mathbb{C}$ be the affine part of $\Sigma_0(p)$,  there is a biholomorphism 
\begin{align}
X_p \cong (\mathcal{S}_0(p) \setminus \mathcal{B}(p) \times ...  \times \mathcal{S}_0(p) \setminus \mathcal{B}(p)) \setminus \delta \times \mathbb{C}^n
\end{align} 
where $\delta$ is the diagonal.  
\end{rem}
To assist reading,  we summarise conventions for indexing various data in a table:
\begin{table}[h]
\caption{Conventions for indices}
\centering
\begin{tabular}{|c|c|c|}
\hline
Data & Index style & Range  \\
\hline
Poles $w_\alpha$ of $\phi$/$Q_0$  & Greek ($\alpha,  \beta,...$) & $\alpha = \infty, 1,...,N$ \\
Laurent coefficients $a_i^{(\alpha)}$ of $Q_0$ poles & Latin ($i,j,...$) & $i = 1,...,2m_\alpha - 1,  \alpha \ne \infty  $  \\ 
& & $i = 0,...,2m_\infty - 7,  \alpha = \infty$ \\
Poles $q_I$ of $Q_1$ (apparent singularities) & Latin ($I,J,...$) & $I = 1,...,n$,  $n = \mathrm{dim} \,  M/2$ \\
e.g.  Periods $z_a$,  fibre coordinates $\theta_a$ & Latin ($a,b,...$) & $a = 1,...,2n$ \\
\hline
\end{tabular}
\end{table} 

We also summarise notation for the complex manifolds and algebraic curves involved:
\normalsize
\begin{table}[h]
\caption{Complex manifolds}
\centering
\begin{tabular}{|c|c|c|}
\hline
Manifold & Dimension & Coordinates \\
\hline 
$M = \operatorname{Quad}(\bar{m})$ & $2n$ & $\{a_{i}^{(\alpha)}\hspace{-1mm},w_\alpha\}$ \\
$X$ & $4n$ & $\{a_{i}^{(\alpha)}\hspace{-1mm},w_\alpha,  q_I,  v_I\}$  \\
\hline
\end{tabular}
\end{table} 
\begin{table}[h]
\caption{Algebraic curves}
\centering
\begin{tabular}{|c|c|c|c|c|}
\hline
Curve & Defining equation & Dependence & Tautological 1-form & Genus \\
\hline 
$\Sigma_0$ & $y_0^2 = Q_0(x)$ & $p \in M$ & $\psi = y_0 dx$ & $n$   \\
$\Sigma$ & $y^2 = Q(x)$ & $\xi \in X,  \hbar \in \mathbb{C}^*$ & $\Psi = y dx$ & $2n$  \\
\hline
\end{tabular}
\end{table} 
\normalsize

\section{Isomonodromic deformations and the cohomology of $\Sigma(\xi,\hbar)$} \label{curve_approach}
By \textit{isomonodromic flow} we will mean a one-parameter family of diffeomorphisms of $X$ preserving the \textit{partial monodromy data} of the equation (\ref{schro}) consisting of the connection matrices and Stokes' multipliers as defined in \cite{M}.  \par We have the following necessary and sufficient condition for a flow to be isomonodromic:
 A flow $(a_{i}^{(\alpha)}(t),w_\alpha(t),  q_I(t),  v_I(t))$ depending  on an auxiliary parameter $t$ is isomonodromic if and only if there exists a matrix $B(x,t)$ of meromorphic functions such that the connection
\begin{equation*}
\nabla = d - \begin{bmatrix} 0 & 1 \\ Q(x,t) & 0 \end{bmatrix} dx - B(x,t)dt 
\end{equation*}
is flat.   This  is a generalisation \cite{U,  JMU} of Schlesinger's result for ODEs with regular singular points. \par  The general solution for  $B(x,t)$ such that $\nabla$ is flat is
\renewcommand*{\arraystretch}{1.5}
\begin{equation*}
B(x,t) =  \begin{bmatrix} -\frac{1}{2}\frac{\partial A}{\partial x} & A \\  AQ -  \frac{1}{2}\frac{\partial^2A }{\partial x^2} & \frac{1}{2}\frac{\partial A}{\partial x} \end{bmatrix} 
\end{equation*}
for some function $A(x,t)$ which satisfies
\begin{equation}\label{fuchs}
-2\frac{\partial Q}{\partial t} = \frac{\partial^3 A}{\partial x^3} - 4Q\frac{\partial A}{\partial x} - 2\frac{\partial Q}{\partial x}A.  
\end{equation}
This is an equation that appears in,  for example,  \cite{BM} and much earlier in the case of Fuchsian singularities \cite{F}.  \par 
We first need to constrain the function $A(x,t)$ for an isomonodromic flow given the definition of the potential $Q(x)$ in \S \ref{potential_setup}.  \begin{lemma}\label{ansatz}
Any meromorphic function $A(x,t)$ defined for all $x \in \mathbb{CP}^1$ satisfying \rm{(\ref{fuchs})} takes the form
\begin{align}\label{A_ansatz}
A(x,t) = \sum_{I=1}^{n} \frac{f_I(t)}{(x-q_I(t))}
\end{align}
for some holomorphic functions $f_I(t)$.  
\end{lemma}
\begin{proof} 
Consider a fixed value of $t \in \mathbb{C}$.  We will suppress the dependence of functions on $t$ in our notation.  Clearly,  $A(x)$ can only have poles at poles of $Q(x)$ otherwise the right-hand side of (\ref{fuchs}) would have a pole where the left-hand side does not.  

If the leading term of the Laurent expansion of $A$ at the point $x = q_I$ is,  up to scaling by a constant,  $(x-q_I)^{-M_I}$,  $M_I \ge 1$,  then equating coefficients of the $(x-q_I)^{-M_I-3}$ term in the Laurent expansion of (\ref{fuchs}) yields
\begin{align}\label{m_relation}
-M_I(M_I+1)(M_I+2) + 3M_I + 3 = 0
\end{align}
and the only positive solution is $M_I = 1$.  So $A$ has at most a simple pole at $x = q_I$.\par 
Next,  assume $A$ has a pole of order $M_\alpha \ge 1$ at $x = w_\alpha$ so 
\begin{align}
A(x) = \sum_{i=-M_\alpha}^{\infty} A_{i}^{(\alpha)}(x-w_\alpha)^{i}
\end{align}
with $A_{-M_\alpha}^{(\alpha)} \ne 0$.  We have
\begin{align}
Q(x) = \frac{a_{2m_\alpha - 1}^{(\alpha)}}{\hbar^2(x-w_\alpha)^{2m_\alpha-1}} + O\Big((x-w_\alpha)^{-(2m_\alpha-2)}\Big)
\end{align} 
and we see the various terms in (\ref{fuchs}) have poles of no greater order than the integers indicated in braces below:
\begin{align}
-2 \hspace{-2mm}\underbrace{\frac{\partial Q}{\partial t}}_{2m_\alpha} = \underbrace{\frac{\partial^3 A}{\partial x^3}}_{M_\alpha + 3} -  \hspace{1mm} 4\underbrace{Q\frac{\partial A}{\partial x} - 2\frac{\partial Q}{\partial x}A}_{M_\alpha+2m_\alpha}.  
\end{align}
We need to consider some different cases depending on the order of the pole of $Q_0(x)$ at $x = w_\alpha$.  \par 
If $m_\alpha = 1$ then inspecting the $(x-w_{\alpha})^{-M_\alpha - 3}$ term of $(\ref{fuchs})$ we get $A_{-M_\alpha}^{(\alpha)}  = 0$ and so $A$ cannot have a pole at $x = w_\alpha$ if $Q_0(x)$ has a simple pole there.  \par If $m_\alpha > 1$ then equating coefficients of the leading term gives 
\begin{align}
4M_\alpha + 2(2m_\alpha-1) = 0
\end{align}
and there are no integers $M_\alpha$ satisfying this equation.  So $A$ does not have a pole at $w_\alpha$.  \par Next,  we check that $A$ has a zero at infinity.  Laurent expand about $x = \infty$ 
\begin{align}
A(x) = \sum_{i=-\infty}^{M_\infty} A_{i}^{(\infty)} x^{i}
\end{align}
and assume $A_{M_\infty}^{(\infty)} \ne 0$ for $M_\infty \ge 0$. 
Near $x = \infty$ the Laurent expansions of the terms in (\ref{fuchs}) contain terms of order no greater than the integers indicated below:
\begin{align}
-2 \hspace{-2mm}\underbrace{\frac{\partial Q}{\partial t}}_{2m_\infty-7} = \underbrace{\frac{\partial^3 A}{\partial x^3}}_{M_\infty - 3} -  \hspace{1mm} 4\underbrace{Q\frac{\partial A}{\partial x} - 2\frac{\partial Q}{\partial x}A}_{2m_\infty + M_\infty - 6}.  
\end{align}
Recall $m_\infty \ge 3$.  The $x^{2m_\infty + M_\infty - 6}$ term gives 
\begin{align}
-4M_\infty - 2(2m_\infty-5) = 0
\end{align}
and there are no integers $M_\infty,  m_\infty$ which satisfy this equation.  \par 
We have shown $A$ is a rational function with poles only at $x = q_i$,  which are simple,  and a zero at infinity.  This implies we may write it like (\ref{A_ansatz}).  \end{proof} \par 
For each $\hbar \in \mathbb{C}^*$ let $L(\hbar) \subseteq TX$ be the distribution spanned by the generators of isomonodromic flows.  
Fix a point $\xi \in X$ and $\hbar \in \mathbb{C}^*$ such that $\Sigma(\xi,\hbar)$ is non-singular.  Substituting $y^2 = Q(x)$ into (\ref{fuchs}) gives a condition on the tautological meromorphic $1$-form $\Psi \coloneqq y \, dx$ on $\Sigma(\xi,\hbar)$ under an isomonodromic deformation.  It is 
\begin{align}\label{schlesinger_tautological}
\frac{dy}{dt}dx = d(A \cdot y) - \frac{1}{4y}\frac{\partial^3 A}{\partial x^3} dx.
\end{align}
The $1$-form $\Psi$ has poles of order $2m_\alpha-2$ at the branch points $w_\alpha$ with vanishing residue (corresponding to the poles of $\psi$),  and $2n$ additional simple poles at the $n$ pairs of points corresponding to $q_I$ with residues $\pm \sqrt{3}/2$.  In particular,  the residues are constant on $X$ and so the left-hand side of (\ref{schlesinger_tautological}) is a meromorphic $1$-form with vanishing residues and represents a cohomology class.  \par 
Then,  taking the cohomology class of both sides
\begin{align}\label{cohomology_condition}
\Bigg[\frac{dy}{dt}dx\Bigg] = -[\kappa(A)] \in H^1(\Sigma(\xi,\hbar),\mathbb{C}),  
\end{align}
where 
\begin{align}
\kappa(A) = \frac{1}{4y}\frac{\partial^3 A}{\partial x^3} dx.
\end{align}
Similarly to \S\ref{M_defn} we use the notation $U(\Psi) \coloneqq U(y)dx$.  Consider now the linear map
\begin{align}
 \mu(\xi,\hbar) : T_{\xi}X &\to H^1(\Sigma(\xi,\hbar),\mathbb{C}) \\
 U &\mapsto [U(\Psi)] \nonumber.
 \end{align}
This map has an invariant definition in terms of the Gauss-Manin connection for the family of curves,  varying the point in $X$ while fixing $\hbar$: $\Psi$ represents an equivalence class of elements of $H^1(\Sigma(\xi,\hbar),\mathbb{C})$ differing by elements that,  due to the residues being constant on $X$,  are parallel for the Gauss-Manin connection.  $[U(\Psi)]$ is the derivative with respect to the Gauss-Manin connection in the direction $U$,  of any of these representatives.  \par
We will see that this map is not an isomorphism.  \par We may define a  $2$-form $\Omega$ on $T_{\xi}X$ by pulling-back the intersection pairing
$$
 \langle \cdot | \cdot \rangle: H^1(\Sigma(\xi,\hbar),\mathbb{C}) \times H^1(\Sigma(\xi,\hbar),\mathbb{C}) \to \mathbb{C}
$$
on the algebraic curve $\Sigma(\xi,\hbar)$ by $\mu(\xi,\hbar)$.  The same argument as Proposition \ref{residue_formula_prop} shows that there is a formula for this $2$-form in terms of a sum of residues.  In particular,  given $U,V \in T_{\xi}X$ we have 
\begin{align}\label{explicit}
\Omega(U,V) = 2\pi i  \hspace{-2mm}  \sum_{x \in \Sigma({\xi,\hbar})}    \underset{x}{\operatorname{Res}}\big(d^{-1} U(\Psi)V(\Psi)\big).
\end{align}
The right-hand side of the formula (\ref{explicit}) is well-defined as a sum of residues at some poles $\{w_\alpha, q_I\}$ of locally defined meromorphic functions for all $\xi \in X$ and $\hbar \in \mathbb{C}^*$.  Therefore,  for each $\hbar \in \mathbb{C}^*$ we may define a holomorphic $2$-form $\Omega$ on $X$ that agrees with the pullback of the intersection pairing $\langle \cdot | \cdot \rangle$ where $\Sigma(\xi,\hbar)$ is non-degenerate.  
In Appendix \ref{generic_non_singularity} we show that $\Sigma(\xi,\hbar)$ is non-degenerate on an open dense subset of $X \times \mathbb{C}^*$.  

\begin{lemma}[Characterisation of isomonodromic flows I]\label{flows_in_kernel}
 Let $L(\hbar)$ be the distribution generating the isomonodromic flows,  then $L(\hbar) \subseteq \operatorname{ker} \Omega$.
\end{lemma}
\begin{proof} The condition $L(\hbar) \subseteq \operatorname{ker} \Omega$ is a closed one so it suffices to show this for $(\xi,\hbar)$ for which the curve $\Sigma(\xi,\hbar)$ is non-singular. 
For such a  $(\xi,\hbar)$ the intersection pairing $\langle \cdot | \cdot \rangle$ is of course non-degenerate.  This implies $\operatorname{ker} \Omega =\mu(\xi,\hbar)^{-1}{( (\im \mu(\xi,\hbar)})^{\perp})$.  We know from (\ref{cohomology_condition}) that the image of the generator of an isomonodromic flow $V$ is the cohomology class $-[\kappa(A)]$ for some $A$ as in Lemma \ref{ansatz}.  We must therefore show that $[\kappa(A)]$ is orthogonal to the image of $\mu(\xi,\hbar)$.  By linearity,  we may check this separately for each $\kappa_I \coloneqq \kappa(A)$ where $A = (x-q_I)^{-1}$,  for $I = 1,...,n$.  \par 
From the apparent singularity condition (\ref{apparent}) at $x = q_J$ it follows
\begin{align}\label{y_expand}
y = \frac{\sqrt{3}}{2(x-q_J)} + \frac{u_J}{\sqrt{3}} + \frac{2u_J^2}{3\sqrt{3}}(x-q_J) + \frac{9P_J(q_J)-4u_J^3}{9\sqrt{3}}(x-q_J)^2 + ...  , 
\end{align}
where $u_J = p_J/\hbar + v_J$ and $P_J(q_J)$ is the coefficient of $(x-q_J)$ in the expansion of $Q(x)$,  the formula for which will not be important.  \par 
It then follows 
\begin{align}
\kappa_I = \Bigg( -\frac{\sqrt{3}}{(x-q_I)^3} + \frac{2u_I}{\sqrt{3}(x-q_I)^2} + \frac{18P_I(q_I)-16u_I^3}{9\sqrt{3}} + ...  \Bigg) dx,  
\end{align}
while $\kappa_I$ has a simple zero at $x = q_J$ for $J \ne I$.  \par 
We will now,   for arbitrary $U \in T_{\xi}X$,  show that the $1$-form $
\rho_I \coloneqq d^{-1} U(\Psi) \kappa_I 
$  has no residues and hence that $\langle U(\Psi), \kappa_I \rangle = 0$ as required.  \par 
We use (\ref{y_expand}) to calculate,  for $U \in T_{\xi}X$ with no $\frac{\partial}{\partial q_J}$ terms for  $J = 1,...,n$,  
\begin{align}
d^{-1} \big(U(\Psi)\big) &= \frac{U(u_J)}{\sqrt{3}}(x-q_J) + \frac{2u_JU(u_J)}{3\sqrt{3}}(x-q_J)^2 + O((x-q_J)^3).   \label{magic_expansion_1}
\end{align}
While,  using that $\frac{\partial u_K}{\partial q_J}$ vanishes for $K \ne J$ we have the expansions

\begin{align}
d^{-1} \bigg(\frac{\partial}{\partial q_J}(\Psi)\bigg)&=  O((x-q_K)^3),  \ K \ne J, \label{magic_expansion_2}\\ 
d^{-1} \bigg(\frac{\partial}{\partial q_J}(\Psi)\bigg)&=  -\frac{\sqrt{3}}{2(x-q_J)} + \Big(\frac{1}{\sqrt{3}}\frac{\partial u_J}{\partial q_J} - \frac{2u_J^2}{3\sqrt{3}}\Big)(x-q_J) + \frac{2u_J}{3\sqrt{3}}\frac{\partial u_J}{\partial q_J}(x-q_J)^2 \nonumber \\ & \ \ \ \ - \frac{(9P_J(q_J) - 4u_J^3)}{9\sqrt{3}}(x-q_J)^2 + O((x-q_J)^3).  \label{magic_expansion_3}
\end{align}
The above expansions imply for arbitrary $U \in T_{\xi}X$  that the $1$-form $\rho_I$
has no residues at the pair of poles corresponding to $x = q_I$.  \par Lastly,  we argue $\rho_I$ has no other poles.  Since $k_I$ has a simple zero at $x = q_J$ for $J \ne I$,  we see from the above formulae that these are not poles of $\rho_I$.  That the poles of $\phi $ lying in $\mathbb{C}$ are not poles of $\rho_I$ follows because the function $d^{-1}(U(\Psi))$ has a pole of order at most $2m_\alpha - 1$ at such a point while $\kappa_I$ has a zero of order $2m_\alpha$.   At $x = \infty$,  $d^{-1}(U(\Psi))$ has a pole of order at most $2m_\infty - 7$ while $\kappa_I$ has a zero of order $2m_\infty$.  \end{proof}

\section{The twistor distribution of isomonodromic flows}\label{flows_calculation}
In the previous section we obtained a necessary condition for a flow to be isomonodromic: that its generator lies in the kernel of $\Omega$.  Here we will show the existence of $2n$ linearly independent isomonodromic flows by proceeding directly from (\ref{fuchs}).  \par
A subset of the isomonodromic flows will be ``trivially" isomonodromic in the sense that they satisfy (\ref{fuchs}) with $A = 0$.  It is convenient to introduce some terminology:
\begin{defn}[Isopotential flows]\label{isopotential_defn}
An isopotential flow is a flow generated by a vector field $L \in \Gamma(TX)$ such that
\begin{align}
L(Q(x)) = 0.  
\end{align}
\end{defn}  
\begin{prop}\label{isopotential_flows_prop}
There are $n$ linearly independent isopotential flows.  They are 
\begin{align}\label{isopotential_flows}
L_{i}^{(\alpha)} \coloneqq \frac{\partial}{\partial a_i^{(\alpha)}} - \frac{1}{\hbar}\sum_{I=1}^{n} \frac{\partial p_I}{\partial a_i^{(\alpha)}} \frac{\partial}{\partial v_I}, 
\end{align}
where $1 \le i \le m_\alpha$ and $\alpha = 1,...,N$ and  $0 \le i \le m_\infty - 4$ for $\alpha = \infty$.
\end{prop}
\begin{proof}
  The locations $w_\alpha,  q_I$ of the poles must be fixed for $\alpha = 1,...,N$ and $I = 1,...,n$ in order for $Q(x)$ to be preserved by a flow.   By comparing the pole orders of $Q_0(x)$ and $Q_1(x)$,  we see $a_i^{(\alpha)}$ must be fixed for $i > m_\alpha$ for $\alpha = 1,...,N$ and $a_i^{(\infty)}$ for $i > m_\infty - 4$.  
Using the apparent singularity condition (\ref{apparent}) an isopotential flow will satisfy 
\begin{align}\label{isopotential_condition}
L(u_I) = 0
\end{align}
for $I = 1,...,n$.  This implies there are at most $n$ isopotential flows since the functions $u_I$ give $n$ independent functions of the $2n$ parameters that are not necessarily held fixed.  Note from the apparent singularity condition that $u_I$ being annihilated by $L$ implies $L(Q(x))$ has $n$ roots at the points $q_I$.  Now $L(Q(x))$ will have,  as a rational function,  a sum of pole orders at most
\begin{align}
m_\infty - 4 + \sum_{\alpha=1}^{N} m_\alpha = n - 1,  
\end{align}
so can have at most $n-1$ roots or is zero.  (\ref{isopotential_condition}) is therefore necessary and sufficient for $L(Q(x)) = 0$.  To see existence of flows satisfying (\ref{isopotential_condition}) note that the $n$ linearly independent vector fields (\ref{isopotential_flows})
satisfy (\ref{isopotential_condition}) where $1 \le i \le m_\alpha$ and $\alpha = 1,...,N$ and  $0 \le i \le m_\infty - 4$ for $\alpha = \infty$.  \end{proof} 

\begin{prop}\label{explicit_flows_calc_prop}
There are $2n$ linearly independent isomonodromic flows on $X$ defining a twistor distribution in the sense of Definition \textup{ \ref{twistor_distribution_def}}.  They are given by isopotential flows $L_{i}^{(\alpha)}$,  for $1 \le i \le m_\alpha$ where $\alpha = 1,...,N$ and $0 \le i \le m_\infty - 4$ for $\alpha = \infty$,  
together with,  for $I = 1,...,n$
\begin{align}
L_I = U_I + \frac{V_I}{\hbar},
\end{align}
with $U_I,V_I$ given in Appendix \ref{explicit_flows}.  

\end{prop}

\begin{proof}
We will find the flows which solve (\ref{fuchs}) with $A = (x-q_I)^{-1}$ for each $I$.  By linearity,  and Lemma \ref{ansatz} this will give all the isomonodromic flows.  The calculation will involve the Laurent expansion of $Q(x)$ at $x = q_I$ to the next order.  Write
\begin{align}
Q'(x) = P_I(x) + W_I(x),  \quad I = 1,...,n,  
\end{align}
where $P_I(x)$ is holomorphic at $x = q_I$ and $W_I(x)$ is the principal part (singular terms) of the Laurent expansion of $Q'(x)$ at $q_I$ so that
\begin{align}
Q(x) = \frac{3}{4(x-q_I)^2} + \frac{u_I}{x-q_I} + u_I^2 + P_I(q_I)\cdot(x-q_I)+ O\big((x-q_I)^2\big).  
\end{align}
We now consider (\ref{fuchs}) with $A = (x-q_I)^{-1}$ and insist that the $(x-q_K)^{l}$ terms,  in the Laurent expansion of both sides of (\ref{fuchs}),  $l \le 0$,  agree for $K = 1,...,n$.  We need to handle the $K = I$ and $K \ne I$ terms separately.  For $K = I$ the conditions are
\begin{equation}
\begin{aligned}
(x-q_I)^{-3}:& \\
(x-q_I)^{-2}:& \\
(x-q_I)^{-1}:& \\
(x-q_I)^{0}\ \, :& 
\end{aligned}
\ \ \ \ \ 
\begin{aligned}
\dot{q}_I &= -2u_I \\
-2\dot{q}_I u_I &= 4u_I^2 \\ 
-2\dot{u}_I &= 2P_I(q_I) \\
-4u_I\dot{u}_I +  2\dot{q}_IP_I(q_I) &= 0.
\end{aligned}
\end{equation}
Note the $(x-q_I)^0$ and $(x-q_I)^{-2}$ equations are implied by the other two.   For $K \ne I$,  a more complicated computation gives
\begin{equation}
\begin{aligned}
(x-q_K)^{-3}:& \\
(x-q_K)^{-2}:& \\
(x-q_K)^{-1}:& \\
(x-q_K)^{0}\ \, :& 
\end{aligned}
\ \ \ \ \ 
\begin{aligned}
\dot{q_K} &= -(q_K-q_I)^{-1} \\
-2\dot{q}_K u_K &= 2u_K(q_K-q_I)^{-1} \\ 
-2\dot{u}_K &= 2u_K(q_K-q_I)^{-2} - 3(q_K-q_I)^{-3} \\
-4u_K\dot{u}_K +  2\dot{q_K}P_K(q_K) &= -\frac{2P_K(q_K)}{q_K-q_I} -\frac{6u_K}{(q_K-q_I)^3} + \frac{4u_K^2}{(q_K-q_I)^{2}}
\end{aligned}
\end{equation}
and similarly the $(x-q_K)^0$ and $(x-q_K)^{-2}$ equations are implied by the other two. \par 

The conditions give expressions for $\dot{q_K}$ and $\dot{u_K}$ for $K = 1,...,n$.
 Altogether 
 \begin{align}
 \dot{q_I} &= -2\Big(\frac{p_I}{\hbar} + v_I\Big), \\
  \dot{q_K} &= -(q_K-q_I)^{-1},  \quad K \ne I  
 \end{align}
 and 
 \begin{align}
 \dot{v_I} &= -P_I(q_I) - \frac{\dot{p_I}}{\hbar}, \label{flow_0} \ \\
 \dot{v_K} &= - \frac{\dot{p_K}}{\hbar} - \Big(\frac{p_K}{\hbar} + v_K\Big)(q_K-q_I)^{-2} + \frac{3}{2}(q_K-q_I)^{-3},  \quad K \ne I. \label{flow_1}
 \end{align}
For $K = 1,...,n$,  a simple application of the chain rule yields:
 \begin{align}
 \dot{p_K} &= \frac{1}{2p_K}\big(Q_0'(q_K)\dot{q_K} + \dot{Q_0}(q_K)\big),
 \end{align}
so that 
 \begin{align}\label{littleP_eqns1}
 \dot{p_I} &= -\frac{Q_0'(q_I)}{\hbar} + \frac{1}{2p_I}\big(\dot{Q_0}(q_I)-2Q_0'(q_I)v_I \big), \\
  \dot{p_K} &=  \frac{1}{2p_K}\Bigg(\dot{Q_0}(q_K) - \frac{Q_0'(q_K)}{q_K-q_I} \Bigg), \quad   K \ne I.\label{littleP_eqns2}
 \end{align}
Meanwhile,  in full detail:
 \begin{align}\label{BigP_def}
P_I(x) = \frac{Q_0'(x)}{\hbar^2} + \frac{R'(x) - \sum\limits_{J\ne I}^n\frac{p_J}{(x-q_J)^2}}{\hbar} - \sum_{J\ne I}^n\frac{3}{2(x-q_J)^3} - \sum_{J \ne I}^n\frac{v_J}{(x-q_J)^2} +S'(x).  
 \end{align}
Substituting (\ref{littleP_eqns1}) through (\ref{BigP_def}) into (\ref{flow_0}) and (\ref{flow_1}) we obtain expressions for $\dot{v_K}$ for $K = 1,...,n$ as functions of $\{a_i^{(\alpha)},  w_\alpha,  v_I,  q_{I}\}$ and $\{\dot{a_i}^{(\alpha)},\dot{w_\alpha}\}$.  \par We will now find expressions for $\{\dot{a_i}^{(\alpha)},\dot{w_\alpha}\}$ in terms of $\{a_i^{(\alpha)},w_\alpha,  v_I,  q_{I}\}$.  \par 
To solve for $\dot{w_\alpha}$ use the $(x-w_\alpha)^{-2m_\alpha}$ term in the Laurent expansion of (\ref{fuchs}):
\begin{align}
-2(2m_\alpha-1){a_{2m_\alpha-1}^{(\alpha)}\dot{w_\alpha}} = \frac{2(2m_\alpha-1)}{w_\alpha - q_I} {a_{2m_\alpha-1}^{(\alpha)}}
\end{align}
so 
\begin{align}
\dot{w_\alpha} = -\frac{1}{w_\alpha-q_I},  \quad \alpha = 1,...,N.  
\end{align}
Next,  the $(x-w_\alpha)^{-i}$ term of (\ref{fuchs}) yields,  after a calculation taking a product of Laurent expansions,  equations for $i > m_\alpha$

\begin{align}
{\dot{a_i}^{(\alpha)}} = {T^{(\alpha)}_i(q_I-w_\alpha)},   
\end{align}
where
\begin{align}
T_i^{(\alpha)}(x) = \sum_{k=i}^{2m_\alpha-1}(2i-k-2)\frac{a_k^{(\alpha)}}{x^{k-i+2}}.  
\end{align}
Similarly,  the terms in the Laurent expansion at infinity give,  for $i > m_\infty - 4$
 \begin{align}
 \dot{a_{i}}^{(\infty)} = T_{i}^{(\infty)}(q_I),  
 \end{align}
 where 
 \begin{align}
 T_{i}^{(\infty)}(x) = (2i-2m_\infty + 7)x^{2m_\infty-i-7} + \sum_{k=i+2}^{2m_\infty - 7}(2i-k+2)x^{k-i-2}a_k^{(\infty)}.  
 \end{align}
 We have now specified $\{\dot{a_i}^{(\alpha)}, \dot{w_\alpha}, \dot{v_I},  \dot{q_I}\}$ in terms of $\{a_i^{(\alpha)}, w_\alpha,  v_I,  q_{I}\}$ in such a way that (\ref{fuchs}) is satisfied up to a rational function with roots at the $n$ points $q_I$ but with poles the orders of which sum to at most
\begin{align}
  m_\infty - 4 + \sum_{\alpha = 1}^N  m_\alpha  = n - 1. 
 \end{align} 
   Thus,  this rational function is zero and (\ref{fuchs}) is satisfied.  We therefore see that (given the freedom to add isopotential flows) there is an $n+1$ dimensional space of flows satisfying (\ref{fuchs}) with $A = (x-q_I)^{-1}$ for each $I$.  We may of course add the unique isopotential flow so that $\dot{a_i}^{(\alpha)}=0$ where $1 \le i \le m_\alpha$ for $\alpha = 1,...,N$ and $0 \le i \le m_\infty - 4$ for $\alpha = \infty$.  We denote the generator of this distinguished flow,  with $A$ as above,  by $L_I$.  \par 
Note that,  substituting everything into (\ref{flow_0}) and (\ref{flow_1}),  the expressions for the time derivatives $\{\dot{a_i}^{(\alpha)}, \dot{w_\alpha}, \dot{v_I},  \dot{q_I}\}$ in terms of $\{a_i^{(\alpha)}, w_\alpha,  v_I,  q_{I}\}$ contain only $\hbar^0$ and $\hbar^{-1}$ terms.  In particular,  we may write
\begin{align}
L_I = U_I + \frac{V_I}{\hbar}.  
\end{align}
Note that the flows are defined where the $q_I$ are distinct and away from the branch points $\mathcal{B}(p)$ so defined on all of $X$.  \par The full set of isomonodromic flows $\{L_i^{(\alpha)}\} \cup \{L_I\}$ are given explicitly in Appendix \ref{explicit_flows}.
It remains to show that the set of $4n$ vector fields
 $$\mathcal{S} = \{U_I,V_I\} \cup \{U_i^{(\alpha)},V_i^{(\alpha)}\} $$ trivialise $TX$ and so define a twistor distribution in the sense of Definition \ref{twistor_distribution_def}.  
 This is essentially a combinatorial argument and so deferred to Appendix \ref{linear_independence_appendix}.  
 \end{proof}
 
\begin{rem}
Not only are the $2n$ vector fields $\{U_{i}^{(\alpha)}\} \cup \{U_I\}$ linearly independent but they span a distribution transverse to $M$.  
\end{rem}

\section{$\Omega$ as an $\mathcal{O}(2)$-valued twistor $2$-form}\label{Omega_is_O2}
The $1$-form $\Psi$ has poles only at $x \in \{\infty, w_1,...,w_N\}$,  corresponding to poles of $\phi$,  and the pairs of poles corresponding to $x = q_I$ for $I = 1,...,n$.  Accordingly,  $\Omega$ may be evaluated from a sum of residues at these points.  
Near the poles of $\phi$ we may expand
\begin{align}
y &= \frac{y_0}{\hbar}\sqrt{1 + \hbar\frac{Q_1(x)}{Q_0(x)} + \hbar^2\frac{Q_2(x)}{Q_0(x)}} \nonumber \\ 
&= \frac{y_0}{\hbar} + \frac{Q_1(x)}{2y_0} + \hbar\frac{4Q_0(x)Q_2(x)-Q_1(x)^2}{8y_0^3} + O(\hbar^2). 
\end{align}
Here we have used the fact the argument of the square root is holomorphic in a neighbourhood of $\hbar = 0$.  Identifying sufficiently small neighbourhoods of the poles of $\phi$ in $\Sigma_0(p)$ and $\Sigma(\xi,\hbar)$,  we may write,  restricting to a small neighbourhood of such a pole (which will be a branch point for both curves), 
\begin{align}\label{hbar_expansion}
\Psi = \frac{\psi}{\hbar} + \sigma + \hbar\tau + O(\hbar^2),  
\end{align}
where we define meromorphic $1$-forms on $\Sigma_0(p)$
\begin{align}
\sigma &\coloneqq \frac{Q_1(x)}{2y_0} dx,  
\\ 
\tau &\coloneqq \frac{4Q_0(x)Q_2(x)-Q_1(x)^2}{8y_0^3}dx. 
\end{align}
The $1$-form $\sigma$ (previously introduced in \S \ref{potential_setup}) has constant residues at simple poles at the pairs of points corresponding to $q_I$ for $I = 1,...,2n$ and no other poles.  In particular,   $U(\sigma) := U(Q_1(x)/2y_0) dx$ represents a class in $H^1(\Sigma_0(p),\mathbb{C})$ for $U \in T_\xi X$.  
 \begin{lemma}\label{hbar3_does_not_contribute}
 Terms of order $\hbar^2$ and higher in the $\hbar$ expansion \textup{(\ref{hbar_expansion})} have zeroes of order at least $2m_\alpha-2$ at $x = w_\alpha$ and $x = \infty$ as $1$-forms while $\sigma,\tau$ are holomorphic in a neighbourhood of these points.  
 \end{lemma}
\begin{proof}
Consider the expansion
\begin{align}
\sqrt{1+t} = 1+\frac{t}{2}-\frac{t^2}{8}+\frac{t^3}{16} + O(t^4)
\end{align}
and substitute $t = \hbar{Q_1(x)}/{Q_0(x)} + \hbar^2{Q_2(x)}/{Q_0(x)}$. 
We see immediately that terms of order $\hbar^k$ in the resulting expansion are polynomials in ${Q_1(x)}/{Q_0(x)}$ and ${Q_2(x)}/{Q_0(x)}$ of degree $k$ with no terms of degree less than $k/2$.  Using a parameter $s_\alpha$ such that $s_\alpha^2 = (x-w_\alpha)$ in a neighbourhood of $w_\alpha$ for $\alpha = 1,...,N$, we see that the functions $Q_1(x)/Q_0(x)$ and ${Q_2(x)}/{Q_0(x)}$ each have zeroes of degree $2m_\alpha - 2$.  Meanwhile $y_0 dx$ has a pole of order $2m_\alpha-2$ at $w_\alpha$ so that the terms in (\ref{hbar_expansion}) of order $\hbar^2$ or greater have zeroes of order at least $2m_\alpha-2$.  A similar argument works at $x = \infty$.  
\end{proof}
We will see that the upshot of Lemma \ref{hbar3_does_not_contribute} is that we can evaluate the contribution to $\Omega$ from residues at the poles of $\phi$ purely from the $1$-forms $\psi,\sigma,\tau$ which are defined on $\Sigma_0(p)$.  The same argument for the validity of the local expansion (\ref{hbar_expansion}) does not work in the neighbourhood of a point corresponding to $x = q_I$ but the required terms in the Laurent expansion to evaluate the contribution of a residue here to $\Omega$ are (see also (\ref{y_expand}))
\begin{align}\label{q_I contribution to Omega}
y = \frac{\sqrt{3}}{2(x-q_I)} +  \hbar^{-1}\frac{p_I}{\sqrt{3}}+ \frac{v_I}{\sqrt{3}} + O(x-q_I).  
\end{align}
\begin{prop}\label{Omega_is_O(2)_valued}
We may write 
\begin{align}
\Omega = \frac{\Omega_{-}}{\hbar^2} + \frac{i\Omega_{I}}{\hbar} + \Omega_{+},  
\end{align}
where $\Omega_{-},\Omega_{I},\Omega_{+}$ are closed $2$-forms on $X$ that do not depend on $\hbar$.  \par Specifically,  for $U,V \in T_\xi X$ 
\begin{align}
\Omega_{-}(U,V) &= 2\pi i \sum_{\alpha} \underset{x=w_\alpha}{\mathrm{Res}} \big( d^{-1} U(\psi) V(\psi) \big) = \omega(U,V), \label{Omega0_formula} \\
i\Omega_{I}(U,V) &= 2\pi i \sum_{I=1}^n dp_I \wedge dq_I  + 2\pi i \sum_{\alpha} \underset{x=w_\alpha}{\mathrm{Res}} \big( d^{-1} U(\psi)  V(\sigma) + d^{-1} U(\sigma) V(\psi)  \big) \label{OmegaI_formula} \\
&= {\langle U(\psi) | V(\sigma) \rangle_0 + \langle U(\sigma) | V(\psi) \rangle_0  } ,\label{OmegaI_intersection_formula}\\ 
\Omega_{+}(U,V) &= 2\pi i \sum_{I=1}^n dv_I \wedge dq_I + 2\pi  i\sum_{\alpha} \underset{x=w_\alpha}{\mathrm{Res}} \big( d^{-1} U(\psi) V(\tau) +  d^{-1} U(\tau) V(\psi)  \big).  \label{Omegainf_formula}
\end{align}
\end{prop}
\begin{proof}
We have 
\begin{align}
\Omega(U,V) = 2\pi i \sum_{\alpha} \hspace{-1mm} \underset{x=w_\alpha}{\text{Res}} \big(  d^{-1} U(\Psi) V(\Psi)\big) + 4 \pi i \sum_{I = 1}^n  \underset{x=q_I}{\text{Res}}\big( d^{-1}  U(\Psi) V(\Psi)).
\end{align}
From (\ref{q_I contribution to Omega}) we see that the second term contributes $\hbar^{-1}$ and $\hbar^{0}$ terms only.  In particular,  we may evaluate,  for $I = 1,...,n$ and $U,V \in T_\xi X$
\begin{align}
 \underset{x=q_I}{\text{Res}}\big( d^{-1} U(\Psi)  V(\Psi)) = \frac{1}{2}\bigg(\bigg(\frac{dp_I}{\hbar} + {dv_I}\bigg) \wedge dq_I  \bigg) \big(U,V\big). 
\end{align}
Meanwhile we have
\begin{align}
&\underset{x=w_\alpha}{\text{Res}} d^{-1} U(\Psi)  V(\Psi) \nonumber
= \underset{x=w_\alpha}{\text{Res}} d^{-1}U\Bigg(\frac{\psi}{\hbar}+{\sigma}+\hbar\tau + \cdots \Bigg) V\Bigg(\frac{\psi}{\hbar}+{\sigma}+\hbar\tau + \cdots \Bigg) \\
  = &\underset{x=w_\alpha}{\text{Res}} \Bigg( d^{-1}U\Bigg(\frac{\psi}{\hbar}+{\sigma}\Bigg)V\Bigg(\frac{\psi}{\hbar}+{\sigma}\Bigg)  + {d^{-1} U(\psi)  V(\tau) + d^{-1}  U(\tau) V(\psi)} \hspace{-1mm} \Bigg),  
\end{align}
since from Lemma \ref{hbar3_does_not_contribute} we know the remaining terms in the expansion correspond to taking the residues of $1$-forms holomorphic at $x = w_\alpha$.  Also note that 
\begin{align}
 \underset{x=w_\alpha}{\mathrm{Res}} \big( d^{-1} U(\sigma) V(\sigma) \big) = 0,  
\end{align}
since $\sigma$ is holomorphic at $x = w_\alpha$.  Defining $\Omega_{-},\Omega_{I},\Omega_{+}$ to be the terms in the expansion of $\Omega$ of appropriate order in $\hbar$ then yields (\ref{Omega0_formula}) through (\ref{Omegainf_formula}).  The closure of these $2$-forms follows immediately from the closure of $\Omega$ for fixed $\hbar$ (which follows from integration by parts akin to Proposition \ref{residue_formula_prop}).  
\end{proof}
Let $\Omega$ define a $2$-form on $X \times \mathbb{CP}^1$ by letting $\hbar = u^1/u^0$ where $[u^0:u^1] \in \mathbb{CP}^1$.  From $L(\hbar) \subseteq \ker \Omega$ (Proposition \ref{flows_in_kernel}) and Proposition \ref{Omega_is_O(2)_valued} above,  we may (Proposition \ref{practical_definition_twistor_form}) interpret $\Omega$ as a closed $\mathcal{O}(2)$-valued relative $2$-form on the twistor space $Z = X \times \mathbb{CP}^1 / L$.  \par 
The $2$-form $\Omega_{-}$ is simply the pullback of $\omega$ defined on the base $M$.  $\Omega_{+}$ is more complicated but simple to describe when restricted to the vectors vertical for the projection to $M$ as the following lemma shows.  
\begin{lemma}\label{vertical_symplectic_form}
Let $VX$ denote the vertical bundle for the projection $X \to M$.  Then 
\begin{align}\label{vert_form}
\Omega|_{VX} = \Omega_{+}|_{VX} = 2\pi i\sum_{I=1}^{n} dv_I \wedge dq_I.  
\end{align}
\end{lemma}
\begin{proof}
Let $U,V \in VX$ then since $U(\psi) = V(\psi) = 0$ the formula (\ref{Omegainf_formula}) reduces to (\ref{vert_form}).  
\end{proof}

\begin{cor}[Characterisation of isomonodromic flows II]
Interpreting $\Omega$ as a $2$-form on $X \times \mathbb{CP}^1$ we have 
\begin{align}
\operatorname{ker} \Omega = L \oplus p_2^*T\mathbb{CP}^1,  
\end{align}
where $L$ is the twistor distribution and $p_2$ is the projection $X \times \mathbb{CP}^1  \to \mathbb{CP}^1$.  
\end{cor}
\begin{proof}
Lemma \ref{flows_in_kernel} states $L \subseteq \operatorname{ker} \Omega$ and Proposition \ref{explicit_flows_calc_prop} implies $L$ is of rank $2n$.  Clearly $p_2^*T\mathbb{CP}^1  \subseteq \ker \Omega$.  
By Lemma \ref{vertical_symplectic_form},  we see that $ \ker \Omega$ is transverse to $VX$ for all $\hbar \in \mathbb{C}^*$ and therefore of rank at most $2n+1$.  
\end{proof}
The above corollary implies that $\Omega$ has maximal rank as a relative $\mathcal{O}(2)$-valued $2$-form on $Z$.  Applying Theorem \ref{twistor_characterisation},  we have proven:
\begin{thm}
$\Omega$ defines a complex hyper-K\"ahler metric $g$ on $X$.  
\end{thm}

\section{Symmetries of $g$ and the Joyce structure} \label{Joyce_finalisation}
Having shown $X$ has a complex hyper-K\"ahler structure,  in this section we will show that we can equip $M$ with a Joyce structure.  First,  we demonstrate that the hyper-K\"ahler metric $g$ constructed above has two required symmetries.  Then,  we construct an open (but not in general injective) immersion from $X$ to an algebraic torus bundle $\mathbb{T} \to M$ that is naturally a quotient of $TM$ by the period lattice.  This is a relatively straightforward generalisation of the \textit{abelian holonomy map} in \cite{BM}.  Equivariance with respect to deck transformations enables us to push forward all the data to open subsets of $\mathbb{T}$ and we show the resulting structure satisfies the symmetries of a Joyce structure (Definition \ref{J}). 
\begin{prop}[Fibrewise involution symmetry]\label{Involution1}
The map $\hat{\iota}: X \to X$ given by $$(a,q,p,v) \mapsto (a,q,-p,v)$$ satisfies
\begin{align}
\hat{\iota}^*I=I,  \quad \hat{\iota}^*(J \pm iK) = -(J \pm iK),  \quad \hat{\iota}^*g = -g.  
\end{align}
\end{prop}
\begin{proof}
The first two equalities follow immediately from the definitions of the complex structures (\ref{quaternions}) and that,  after inspecting the formulae for the isomonodromic flows in Appendix \ref{explicit_flows},  the vector fields $V_I,  V_i^{(\alpha)}$ are odd under simultaneously changing the signs of each $p_J$ while the $U_I,  U_i^{(\alpha)}$ remain unchanged.  \par For the last equality note that $\psi$ is invariant under $\hat{\iota}$ since it depends only on the projection to $M$ while inspecting the formula for $Q_1$ we see $\sigma \mapsto -\sigma$.  The expansion (\ref{q_I contribution to Omega}) shows that the $\hbar^{-1}$ term in the contribution to $\Omega$ from residues at points corresponding to $x = q_I$ is odd.   Accordingly,  we see that $\hat{\iota}^*\Omega_I = -\Omega_I$ which,  using the definition $\Omega_{I} = g(I\cdot,\cdot)$ together with the first equality shows the last equality.  
\end{proof}
To demonstrate the hyper-K\"ahler structure has the appropriate homogeneity properties,  the following will be helpful: 
\begin{lemma}\label{homogeneity_lemma}
A complex hyper-K\"ahler structure admits a vector field $E$ satisfying 
\begin{align}\label{twistor_homogeneity}
\mathcal{L}_E\Omega_{-} = 2\Omega_{-},  \quad \mathcal{L}_E\Omega_{I} = \Omega_{I},  \quad \mathcal{L}_E\Omega_{+} = 0,  
\end{align}
if and only if 
\begin{align}\label{homogeneity_E}
\mathcal{L}_EI = 0,  \quad \mathcal{L}_{E}(J \pm iK) = \mp(J \pm iK),  \quad \mathcal{L}_Eg = g.  
\end{align}
\end{lemma}
\begin{proof}
That (\ref{homogeneity_E}) implies (\ref{twistor_homogeneity}) is obvious from the Leibniz rule and the definitions of the $2$-forms.  Conversely,  suppose (\ref{twistor_homogeneity}) holds.  
Using the Leibniz rule and the definitions of the relevant $2$-forms we have equations:
\begin{align}
g(\mathcal{L}_EI\cdot,\cdot) + (\mathcal{L}_Eg)(I\cdot,\cdot) &= g(I\cdot,\cdot),  \label{I_sym} \\ 
g(\mathcal{L}_EJ\cdot, \cdot) + (\mathcal{L}_Eg)(J\cdot,\cdot) - i(g(\mathcal{L}_EK\cdot,\cdot) + (\mathcal{L}_Eg)(K\cdot,\cdot)) &= 2(g(J\cdot,\cdot)-ig(K\cdot,\cdot)), \\
g(\mathcal{L}_EJ\cdot, \cdot) + (\mathcal{L}_Eg)(J\cdot,\cdot) + i(g(\mathcal{L}_EK\cdot,\cdot) + (\mathcal{L}_Eg)(K\cdot,\cdot)) &= 0. 
\end{align}
Rewrite the last two equations as
\begin{align}
g(\mathcal{L}_EJ\cdot,\cdot) + (\mathcal{L}_Eg)(J\cdot,\cdot) &= g(J\cdot,\cdot) - ig(K\cdot,\cdot),  \\
g(\mathcal{L}_EK\cdot,\cdot) + (\mathcal{L}_Eg)(K\cdot,\cdot) &= i(g(J\cdot,\cdot) - ig(K\cdot,\cdot)).  \label{K_eqn}
\end{align}
Precompose the first and second equations with $K$ and $J$ respectively,  then add and subtract the two equations to obtain
\begin{align}
g((\mathcal{L}_EJ)K\cdot,\cdot) + g((\mathcal{L}_EK)J\cdot,\cdot) &= 0,  \\
g((\mathcal{L}_EJ)K\cdot,\cdot) - g((\mathcal{L}_EK)J\cdot,\cdot) &= -2(\mathcal{L}_Eg)(I\cdot,\cdot) + 2g(I\cdot,\cdot) + 2ig(\cdot,\cdot).
\end{align}
Putting the above together with (\ref{I_sym}) and $\mathcal{L}_EI = (\mathcal{L}_EJ)K + J(\mathcal{L}_EK)$,  we see
\begin{align}
2ig(\cdot,\cdot) + 2g(J(\mathcal{L}_EK),\cdot) = 0
\end{align}
so that $\mathcal{L}_EK = iJ$.  Substituting this into (\ref{K_eqn}) gives
\begin{align}
(\mathcal{L}_Eg)(K\cdot,\cdot) = g(K\cdot,\cdot)
\end{align}
and hence,  after precomposing with $K$,  we obtain $\mathcal{L}_Eg = g$.  
Using this,  it is straightforward,  again using the Leibniz rule,  to verify the rest of (\ref{homogeneity_E}). 
\end{proof}
\begin{prop}[Homothetic symmetry]\label{Euler1}
The vector field $E$ on $X$ given by 
\begin{align}
\nonumber E = &\sum_{i = 0}^{2m_\infty - 7} \frac{4m_\infty-10-2i}{2m_\infty-3} \, a_i^{(\infty)} \frac{\partial}{\partial a_{i}^{(\infty)}} + \sum_{\alpha=1}^{N}\sum_{i=1}^{2m_\alpha-1} \frac{4m_\infty-10+2i}{2m_\infty-3} \,  a_i^{(\alpha)}\frac{\partial}{\partial a_{i}^{(\alpha)}} \\ + & \ \, \sum_{\alpha=1}^{N}\frac{2}{2m_\infty-3} \,  w_\alpha \frac{\partial}{\partial w_\alpha} + \sum_{I=1}^{n} \frac{2}{2m_\infty-3}q_I \frac{\partial}{\partial q_I} -  \sum_{I=1}^{n} \frac{2}{2m_\infty-3}v_I \frac{\partial}{\partial v_I}.  
\end{align}
satisfies \textup{(\ref{homogeneity_E})}.  
\end{prop}
\begin{proof}
By inspecting the homogeneity of $Q_0(x),Q_1(x),Q_2(x)$ with respect to the parameters,  we see
\begin{align}
\Bigg( E + \frac{2}{2m_\infty -3}x\frac{\partial}{\partial x} + \hbar\frac{\partial}{\partial \hbar} \Bigg)Q(x) = -\frac{4}{2m_\infty-3}Q(x).  
\end{align}
On the open region where $\Sigma(\xi,\hbar)$ is non-singular we rewrite this equality in terms of the tautological $1$-form $\Psi$ as:
\begin{align}
E(\Psi) + \frac{1}{2m_\infty-3}\frac{x Q'(x)}{y}dx - \frac{1}{y}\Bigg(\frac{Q_0(x)}{\hbar^2}+\frac{Q_1(x)}{2\hbar} \Bigg)dx = -\frac{2}{2m_\infty-3}ydx.  
\end{align}
Integrating by parts using $d(xy) = ydx + xQ'(x)/2y$ yields
\begin{align}
[E(\Psi)] = \Bigg[\frac{1}{y}\Bigg(\frac{Q_0(x)}{\hbar^2}+\frac{Q_1(x)}{2\hbar} \Bigg)dx\Bigg] =: [\Phi]. 
\end{align}
Treat $\hbar$ as a fixed constant and use Cartan's formula:
\begin{align}
(\mathcal{L}_E\Omega)(U,V) &= d(\Omega(E,\cdot))(U,V)  \nonumber \\
&= U(\Omega(E,V)) - V(\Omega(E,U)) - \Omega(E,[U,V]) \nonumber \\
&= U( \langle \Phi | V(\Psi) \rangle) - V (\langle \Phi | U(\Psi) \rangle ) - \langle \Phi | [U,V](\Psi)\rangle \nonumber \\
&= \langle U(\Phi) | V(\Psi) \rangle - \langle V(\Phi) | U(\Psi) \rangle.  
\end{align}
Note that 
\begin{align}
\Phi = \hbar \frac{\partial \Psi}{\partial \hbar} dx.  
\end{align}
We combine this observation with (\ref{hbar_expansion}).   
Near the poles of $Q_0(x)$,  expand in orders of $\hbar$
\begin{align}
\frac{1}{y}\Bigg(\frac{Q_0(x)}{\hbar^2}+\frac{Q_1(x)}{2\hbar} \Bigg) dx = \frac{\psi}{\hbar} - \hbar\tau + O(\hbar^2),  
\end{align}
where the omitted terms have  (see Lemma (\ref{hbar3_does_not_contribute})) zeroes of order at least $2m_\alpha - 2$ at $x = w_\alpha$.  
Similarly,  from (\ref{q_I contribution to Omega}) we have the Laurent expansion near $x = q_I$:
\begin{align}
\frac{1}{y}\Bigg(\frac{Q_0(x)}{\hbar^2}+\frac{Q_1(x)}{2\hbar} \Bigg) = \hbar^{-1}\frac{p_I}{\sqrt{3}} + O(x-q_I).  
\end{align}
Then
\begin{align}
&\langle U(\Phi) | V(\Psi) \rangle  = 2\pi i \sum_{\alpha} \hspace{-1mm} \underset{x=w_\alpha}{\text{Res}} \big( d^{-1} U(\Phi) V(\Psi)\big) + 4 \pi i \sum_{I = 1}^n  \underset{x=q_I}{\text{Res}} ( d^{-1} U(\Phi)  V(\Psi)) \nonumber \\ &= 2\pi i \sum_{\alpha} \underset{x=w_\alpha}{\mathrm{Res}} \bigg( \hspace{-1mm} d^{-1}U\bigg(\frac{\psi}{\hbar} - \hbar\tau \hspace{-1mm} \bigg) V\bigg(\frac{\psi}{\hbar} + \sigma + \hbar\tau \hspace{-1mm}  \bigg)  \hspace{-1.5mm} \bigg) + 2\pi i \sum_{I=1}^n  \frac{dp_I(U)dq_I(V)}{\hbar}  + \cdots \hspace{-0.5mm} ,   
\end{align}
where the omitted terms are proportional to $dq_I(U)dq_I(V)$ and do not contribute to the result of anti-symmetrising the above in $U$ and $V$.  Comparing to (\ref{Omega0_formula}) and (\ref{OmegaI_formula}) we obtain the following equality,  which holds on an open dense region of $X$ and hence on $X$:
\begin{align}\label{twistor_homogeneity_h}
(\mathcal{L}_E\Omega)(U,V) = \langle U(\Phi) | V(\Psi) \rangle - \langle V(\Phi) | U(\Psi) \rangle = \frac{2\Omega_{-}(U,V)}{\hbar^2} + \frac{2i\Omega_{I}(U,V)}{\hbar}.
\end{align}
That this holds to all orders in $\hbar$ is equivalent to (\ref{twistor_homogeneity}),  which implies (\ref{homogeneity_E}). 
\end{proof}
Next,  we define a map $\Theta: X \to \mathbb{T}$ where $\mathbb{T} = TM / \mathcal{G}$.  Recall that $\mathcal{G}$ is the lattice generated by integer linear combinations of fundamental (co)cycles pulled back by the isomorphism $\mu_0: TM \to \mathcal{E}$ induced by the Gauss-Manin connection. \par
$\mathbb{T}$ is equipped with \textit{period coordinates} $\{z^{a},\theta^{a}\}_{a=1}^{2n}$, where $\{z^a\}_{a=1}^{2n}$ are the period coordinates on $M$ and given $p \in M$ and $\omega \in H^1(\Sigma_0(p),\mathbb{C}) \cong T_pM$ 
\begin{align}
\theta^{a} = \oint_{\gamma_{a}} \omega \ \,  \mod 2\pi i.  
\end{align}
$\theta^{a}$ is the usual fibre (tangent) coordinate associated with $z^{a}$ in the sense that
\begin{align}
\theta^{a} \Bigg( \sum_{b=1}^{2n}V^b\frac{\partial}{\partial z^{b}}\Bigg) = \oint_{\gamma_{a}} \sum_{b=1}^{2n} V^b \frac{\partial y_0}{\partial z^b} dx = \sum_{b=1}^{2n} V^b\frac{\partial z^a}{\partial z^b} = V^a.  
\end{align}

For $a = 1,...,n$ we let $A_{a} \coloneqq \gamma_{a}$ (the $A$-cycles) and $B_{a} \coloneqq \gamma_{a+n}$ (the $B$-cycles).  \par
\begin{prop}
The map $\Theta: X \to \mathbb{T}$ over $M$ given in period coordinates by 
\begin{align}
\theta^{a} = -\oint_{\gamma_{a}} \varpi
\end{align}
for $a = 1,...,2n$ where 
\begin{align}
\varpi = \sigma + 
\sum_{I=1}^{n}\frac{dx}{2(x-q_I)}
\end{align}
is an immersion with image an open dense subset of $\mathbb{T}$.  
\end{prop}
\begin{proof}
Fix $p \in M$.  Letting the fibre coordinates vary 
\begin{align}
\varpi = \sigma + \sum_{I=1}^{n}\frac{dx}{2(x-q_I)} = \bigg( \sum_{I=1}^{n}\frac{ y + p_I}{x-q_I} + R(x)\bigg)\frac{dx}{2y}
\end{align}
is the general one-form with simple poles with residue $1$ at the $n$ points $(q_I, p_I)$,  residue $-n$ at $\infty$,  and no other poles.  Therefore,  the functions $\theta^{a}$ are defined up to addition by $2\pi ik$ for $k \in \mathbb{Z}$.  There are unique $t_{a} \in \mathbb{C}$ for $a = 1,...,n$ such that
\begin{align}
\varpi = \sum_{I=1}^n \nu_I + \sum_{a=1}^n t_a \omega_a.
\end{align}
Here $\nu_I$ and $\omega_a$ are normalised meromorphic $1$-forms as in \cite{Bo2}:  The normalisation condition is that $\nu_I $ has simple poles with residues $\pm 1$ at points $(q_I,p_I)$ and $\infty$,  respectively,  and we have chosen a set of representative $A$-cycles for which the periods vanish.  The $\omega_{a}$ are holomorphic,  normalised so 
\begin{align}
\oint_{A_{b}} \omega_{a} = 2\pi i \delta_{ab}.  
\end{align}
  We have the following identities (\cite{Bo2} Section 4.3),  for any set of representative cycles:
\begin{align}
\oint_{A_{b}} \sum_{I=1}^n \nu_I &= 0 \ \,  \operatorname{mod} 2\pi i, \\
\oint_{B_{b}} \sum_{I=1}^n  \nu_I &= \sum_{I=1}^{n}  \int_{\infty}^{(q_I,p_I)} \omega_b  \ \,  \operatorname{mod} 2\pi i.  
\end{align}
We now calculate the map $\Theta$.  We have,  for $a = 1,...,n$
\begin{align}
-\theta^a = \oint_{A_a} \varpi = 2\pi i t_a \, \  \operatorname{mod} 2\pi i.  
\end{align}
Then
\begin{align}
-\theta^{a+n} = \oint_{B_{a}} \varpi = \sum_{b=1}^{n}  B_{ab}t_b +  \sum_{I=1}^{n} \int_{\infty}^{(q_I,p_I)} \omega_a   \ \,  \operatorname{mod} 2 \pi i
\end{align}
where $B_{ab}$ is the matrix of $B$-periods for the basis of holomorphic $1$-forms $\{\omega_{a}\}_{a=1}^{n}$.  Now,  the second term on the right-hand side is precisely the $a$th components of the Abel-Jacobi map on the divisor depending on $\xi \in X_p$
\begin{align}\label{q_divisors}
D(\xi) = \sum_{I=1}^n ((q_I,p_I) - \infty).  
\end{align}

We now need some standard facts about the Abel-Jacobi map and the reduced representations of divisors.  Comprehensive references are \cite{ACGH, Mu}.  
The Abel-Jacobi map is a biholomorphism from the group $\operatorname{Pic}_0(\Sigma_0(p))$ of degree-zero divisors modulo principal divisors to the Jacobian variety $J(\Sigma_0(p)) = \mathbb{C}^{n} / \Gamma$ where $\Gamma$ is the integer lattice in $\mathbb{C}^{n}$ generated by the $2n$-vectors $e_a,  \sum_b B_{ab}e_b \in \mathbb{C}^n$ from $a = 1,...,n$.  The Abel-Jacobi map is surjective when restricted to classes represented by divisors $\sum_{I=1}^n (r_I - \infty)$ for points $r_I \in \Sigma_0(p)$ not necessarily distinct.
Recall the conditions on the points $(q_I,p_I)$ was that the $q_I$ were distinct and away from the branch points.  This gives an open dense subset of $\text{Pic}_0(\Sigma_0(p))$.  Accordingly,  $\Theta$ has image an open dense subset of a torus bundle $\mathbb{T} \to M$.  The differential of the Abel-Jacobi map has corank that is the dimension of the space of holomorphic $1$-forms vanishing at the points $(q_I,p_I)$,  which may be calculated as,  using the Riemann-Roch theorem,  $h^0(\Sigma_0(p),\mathcal{O}(\tilde{D}(\xi))) - 1$ where $\tilde{D}(\xi) = \sum_{I=1}^{n} (q_I,p_I)$.  Now we use the standard fact that a divisor of this form consisting of distinct points is special if and only if it contains points interchanged by the covering involution.  So $\tilde{D}(\xi) $ is not special and $h^0(\Sigma_0(p),\mathcal{O}(\tilde{D}(\xi))) = 1$.  It follows that  $\Theta$ is an immersion.  
\end{proof}
The preimage of a point in $\mathbb{T}$ consists of the points in $X$ corresponding to the action of the symmetric group permuting the points $q_I$.  Since the $2$-form $\Omega$ and distribution $L(\hbar)$ are preserved under this action,  we may push them forward to obtain well-defined data on $\mathbb{T}$.  We will abuse notation and use the same symbols to denote this data.  \par
\begin{rem}
The construction of the map $\Theta$ for the case of $\bar{m} = \{7\}$ is given in \cite{BM}.  There,  points in $X$ are interpreted as specifying holomorphic line bundles $\mathcal{O}_{\Sigma_0(p)}(D(\xi))$ together with a connection defined on the bundles by the $1$-form $\varpi$.  The holonomy of the connection is given by the periods of $\varpi$.  
\end{rem} 
\begin{prop}\label{Involution2}
The map $\hat{\iota}: X \to X$ induces $\iota: \mathbb{T} \to \mathbb{T}$ given by 
\begin{align}
(z,\theta) \mapsto (z,-\theta).  
\end{align}
\end{prop}
\begin{proof}
Under $(a,q,p,v) \mapsto (a,q,-p,v)$ we have $\sigma \mapsto -\sigma$.  The differential form 
\begin{align}
\sum_{I=1}^{n}\frac{dx}{2(x-q_I)}
\end{align}
has residues $\pm 1/2$ and is a pullback of a $1$-form on $\mathbb{CP}^1$ so has periods that are multiples of $\pi i$.  So given a fixed class $\gamma_{a}$ we have
\begin{align}
&\theta^{a} \circ \hat{\iota}  = \pi i k + \oint_{\gamma_{a}} \sigma  = -\bigg(\pi i k - \oint_{\gamma_{a}} \sigma \bigg) + 2\pi i k = -\theta^{a} + 2 \pi i k 
\end{align} 
for some $k \in \mathbb{Z}$.  
\end{proof}
\begin{prop}\label{Euler2}
In the period coordinates $\{z^a,\theta^a\}_{a=1}^{2n}$ the push-forward of $E$ is given by
\begin{align}
E = \sum_{a=1}^{2n} z^a\frac{\partial}{\partial z^a}. 
\end{align}
\end{prop}
\begin{proof}
We have the following equalities
\begin{align}
E(Q_0(x)) + \frac{2}{2m_\infty-3}x\frac{\partial}{\partial x}Q_0(x) = \frac{4m_\infty-10}{2m_\infty-3}Q_0(x), \\
E(Q_1(x)) + \frac{2}{2m_\infty-3}x\frac{\partial}{\partial x}Q_1(x) = \frac{2m_\infty-7}{2m_\infty-3}Q_1(x). \label{Q1homogeneity}
\end{align}
The first equality can be rewritten
\begin{align}
E(y_0)dx = \frac{2m_\infty-5}{2m_\infty -3} \, y_0 dx - \frac{1}{2m_\infty-3} \frac{xQ'_0(x)}{y_0} \, dx. 
\end{align}
Integrating the last term by parts yields $[E(\psi)] = [\psi] \in H^1(\Sigma_0(p),\mathbb{C})$.  In particular,   
\begin{align}
E(z^a) = \int_{\gamma_{a}} E(\psi) = z^a
\end{align}
for each $a = 1,...,2n$ so 
\begin{align}
E = \sum_{a=1}^{2n} z^a\frac{\partial}{\partial z^a} + E_V,  
\end{align}
where $E_V$ is a vertical vector field.  Next,  we use (\ref{Q1homogeneity}) which yields
\begin{align}
E(\sigma) &= \frac{E(Q_1(x))}{2y_0}dx - \frac{Q_1(x)E(y_0)}{2Q_0(x)}dx \nonumber \\ &= \Bigg(-\frac{2}{2m_\infty-3}\frac{Q_1(x)}{2y_0} - \frac{2x}{2m_\infty-3}\frac{Q_1'(x)}{2y_0} + \frac{x}{2m_\infty-3}\frac{Q_1(x)Q_0'(x)}{2Q_0^{3/2}}\Bigg)dx \nonumber \\
&= -\frac{2}{2m_\infty-3}d\Bigg(x\frac{Q_1(x)}{2y_0}\Bigg). 
\end{align}
 In particular,   
\begin{align}
E(\theta^a) = -\int_{\gamma_{a}} E(\sigma) = 0
\end{align}
for each $a = 1,...,2n$ so $E_V = 0$.  
\end{proof}
\begin{prop}\label{Omega_I_period}
In the period coordinates $\{z^a,\theta^a\}_{a=1}^{2n}$ the push-forward of the $2$-form $\Omega_I$ is given by
\begin{align}
i\Omega_I = -\sum_{a=1}^{2n} \sum_{b=1}^{2n} \omega_{ab}dz^{a} \wedge d\theta^{b}.  
\end{align}
\end{prop}
\begin{proof}
This follows from the formula (\ref{OmegaI_intersection_formula}) together with 
\begin{align}
dz^{a} = \bigg(\oint_{\gamma_{a}} \psi \bigg) ,  \quad d\theta^{a} = d\bigg(-\oint_{\gamma_{a}} \sigma \bigg) 
\end{align}
and the Riemann bilinear relations.  
\end{proof}
Finally we are able to conclude:
\begin{thm}[Joyce structure on $M$]
The complex manifold $M = \operatorname{Quad}(\bar{m})$ carries a Joyce structure compatible with the standard period structure,  which is possibly singular on the complement of an open dense subset of $TM$.  
\end{thm}
\begin{proof}
We can pull $\Omega$ (constituting the hyper-K\"ahler structure) back by the map $TM \to \mathbb{T}$ given by taking the quotient by the lattice.  We will see this gives a Joyce structure after we check the conditions (1) to (6) in Definition \ref{J}.  \par 
The condition (1) is just the statement the formula (\ref{Omega0_formula}) holds and $\Omega_{-}$ is compatible with $\omega$.  Now,  given an isomonodromic flow $L = U + V/\hbar$,  from the definition (\ref{quaternions}) of the quaternionic structure we have $(J-iK)(V) = 0$ and $(J-iK)(U) = 2V$.  If we let $U_a$ be the unique isomonodromic flow that pushes down to the $z_{a}$ coordinate vector field (we can do this since the isomonodromic flows span a distribution transverse to the vertical bundle for all $\hbar \in \mathbb{C}^*$) then by Proposition \ref{Omega_I_period} 
\begin{align}
\Omega\bigg(\hspace{-1mm}U_a+\frac{V_a}{\hbar},\cdot\hspace{-1mm}\bigg) = \sum_{b=1}^{2n} \bigg(  \frac{\omega_{ab}dz^{b}}{\hbar^2} - \frac{\omega_{ab}d\theta^{b} - \omega_{cb}d\theta^{b}\big(U_a+\frac{V_a}{\hbar}\big)dz^{c}}{\hbar} \bigg) + O(\hbar^0) = 0
\end{align}
so that from the $\hbar^{-2}$ term
\begin{align}
V_a = \frac{\partial}{\partial \theta^{a}}
\end{align}
and hence
\begin{align}
(J-iK)/2 = \sum_{a=1}^{n}  \frac{\partial}{\partial \theta^{a}} \otimes dz^a = \nu
\end{align}
as required for (2).  For (3) we have from (\ref{quaternions}),  $I(U_a) = -iU_a$ so that
\begin{align}
d\pi(U_a) = \frac{\partial}{\partial z^{a}} = i(d\pi(I(U_a))).
\end{align} 
Condition (4) is shown by Proposition \ref{Involution1} and Proposition \ref{Involution2} while Condition (5) follows from Proposition \ref{Euler1} and Proposition \ref{Euler2}.  
Condition (6) is immediate from the fact the hyper-K\"ahler structure is defined on $\mathbb{T}$ so is periodic in $\theta_{a}$ when pulled back to $TM$.  
\end{proof}

\section{Four simple poles and Painlev\'e VI} \label{PVI}
In this section we compute the hyper-K\"ahler metric arising from an example of interest due to its similarity to the isomonodromy problem that leads to the Painlev\'e VI equation \cite{O}.  The moduli space of quadratic differentials we consider here $M = \text{Quad}(\{1,1,1,1\})$,  falls outside those analysed in \S \ref{M_defn},  as we coordinatise $M$ in a slightly modified fashion.  The calculations proceed with only minor adjustments and we therefore omit many details.  \par 
Given a quadratic differential on $\mathbb{CP}^1$ with four simple poles,  we may move three of the poles to the points $\infty,  0,  1$ via M\"obius transformation.  It then takes the form $Q_0(x)dx^2$ where
\begin{align}
Q_0(x) = \frac{a_0}{x} + \frac{a_1}{x-1} +\frac{a}{x-w}. 
\end{align}
The condition that the pole at $x = \infty$ is simple forces $a_0 + a_1 + a = 0$ and $a_1 + aw = 0$ so
\begin{align}\label{PVIQ0}
Q_0(x) = \frac{a(w-1)}{x} - \frac{aw}{x-1} +\frac{a}{x-w}.  
\end{align}
 We may then use $(a,w)$ as local coordinates on $M = \operatorname{Quad}(\{1,1,1,1\})$.  The condition that the zeroes are simple and therefore that $y_0^2 = Q_0(x)$  define a non-singular elliptic curve is equivalent in this case to there being exactly four distinct poles.  That is
 \begin{align}
 aw(w-1) \ne 0.  
 \end{align}
 The earlier arguments of \S \ref{M_defn} through \S \ref{Joyce_finalisation} need to be modified slightly to compute the Joyce structure.  The map $\mu_0$ is defined by differentiating the tautological section with the Gauss-Manin connection as before.  By looking at the order of the poles,  one may see that only residues at $x = w$ will contribute to $\omega$,  the $2$-form induced by the intersection pairings on the elliptic curves.  
 We have the Puiseux expansion

\begin{align}
 y_0 &= a^{1/2}(x-w)^{-1/2} +  O((x-w)^{1/2})
 \end{align}
 so that
 \begin{align}
 \omega = \frac{1}{2}dw \wedge da.  
 \end{align}
and $\mu_0$ is an isomorphism.  \par
The deformation term is given by 
\begin{align}\label{PVIQ1}
Q_1(x) \coloneqq \frac{p}{x-q} + R(x),  
\end{align}
where 
\begin{align}
R(x) = \frac{p(q-1)+r(w-1)}{x} - \frac{pq+rw}{x-1} + \frac{r}{x-w}
\end{align}
in order that $\sigma$ defined by (\ref{sigma}) is the general meromorphic $1$-form with simple poles with residue $\pm 1/2$ at the pair of points corresponding to $x = q$ and no other poles.  
We then choose $Q_2(x)$ uniquely so that $Q(x)$ satisfies the apparent singularity condition (\ref{apparent}) and has simple poles at the four poles of $Q_0(x)$.  \par 
As before we will find it convenient to work with a parameter $v$ instead of $r$ so that $u = p/\hbar + v$ in which case
\begin{align}
R(x) = \frac{-px^2+p((1+w)-q)x + 2pq^2 - pq(1+w)+ 2pvq(q-1)(q-w)}{x(x-1)(x-w)}
\end{align}
so $R(q) = 2pv$ and then 
\begin{align}\label{PVIQ2}
Q_2(x) = \frac{3}{4(x-q)} + \frac{v}{(x-q)} + S(x),
\end{align}
where
\begin{align}
S(x) = \frac{-vx^2 + (v(1+w-q)-\frac{3}{4})x + v^2q(q-1)(q-w) + q(v(2q -w - 1) + \frac{3}{4})}{x(x-1)(x-w)}
\end{align}
so $S(q) = v^2$.  \par
Given this potential,  there are two linearly independent isomonodromic flows for (\ref{schro}).  The vector field
\begin{align}
U &= \nonumber \frac{\partial}{\partial a} - \frac{1}{\hbar} \frac{\partial p}{\partial a}\frac{\partial}{\partial v}  \\
&= \frac{\partial}{\partial a}  - \frac{1}{2\hbar p}  \bigg(\frac{w-1}{q} - \frac{w}{q-1} + \frac{1}{q-w} \bigg) \frac{\partial}{\partial v}  
\end{align}
generates a family of deformations which preserves the potential $Q(x)$.  To calculate the other linearly independent flow we need to make a slight modification to our earlier arguments.  
\begin{lemma}\label{PVIansatz}
Any meromorphic function $A(x,t)$ defined for all $x \in \mathbb{CP}^1$ satisfying \rm{(\ref{fuchs})} with $Q(x) = \hbar^{-2}Q_0(x) + \hbar^{-1}Q_1(x) + Q_2(x)$ specified by (\ref{PVIQ0}),  (\ref{PVIQ1}),  (\ref{PVIQ2}) is given by
\begin{align}\label{A_PVIansatz}
A(x,t) = \frac{f(t)x(x-1)}{x-q(t)}
\end{align}
for some meromorphic function $f(t)$.  
\end{lemma}
\begin{proof}
We will suppress the dependence of functions on $t$ in our notation.  
The same arguments as Lemma \ref{ansatz} imply that $A$ has at most a simple pole at $x = q$ and no poles at $x = 0,1,w$ or anywhere else in $\mathbb{C}$.  Suppose $A(x) = A^{(0)}_0 + A^{(0)}_1x + O(x^2)$ at $x = 0$ then since the pole is fixed the left-hand side (\ref{fuchs}) has a vanishing $x^{-2}$ term.  Setting the $x^{-2}$ term of the right-hand side to zero yields $A^{(0)}_0 = 0$.  This implies $A$ has a zero at $x = 0$.  Similarly,  $A$ has a zero at $x = 1$.  \par
Next,  we consider the behaviour at $x = \infty$.  Note that $Q$ has a zero of order three at infinity.  Write $A(x) = A_{M_\infty}^{(\infty)}x^{M_\infty} + O(x^{M_\infty-1})$.  The leading orders of the terms in (\ref{fuchs}) are at most the integers indicated below 
\begin{align}\label{Ainfty}
-2 \hspace{-2mm}\underbrace{\frac{\partial Q}{\partial t}}_{-3} = \underbrace{\frac{\partial^3 A}{\partial x^3}}_{M_\infty -3} -  \hspace{1mm} 4\underbrace{Q\frac{\partial A}{\partial x} - 2\frac{\partial Q}{\partial x}A}_{M_\infty - 4}.  
\end{align}
In particular,  we can rule out $M_\infty \ge 3$ so that,  the first term on the right-hand side of (\ref{Ainfty}) will have order at most $-4$.  Then,  for $M_\infty = 2$ the last two terms are the leading order terms and their Laurent expansion taken together gives $A_{2}^{(\infty)} = 0$ 
so $M_\infty \le 1$.  In particular,  $A$ has at most a simple pole at $x = \infty$.  All these facts together imply $A$ takes the form (\ref{A_PVIansatz}) up to scale.  
\end{proof}
Similar calculations as in \S \ref{flows_calculation} then imply that the unique flow satisfying (\ref{fuchs}) with $A$ as above and $\dot{a} = 0$ is generated by (up to scale)
\begin{align}\label{PVIgenerator}
&V = \frac{\partial}{\partial w} + \frac{1}{A(w)}\Bigg(\kappa\frac{\partial}{\partial q} + \Big(q(q-1)S'(q)-v\Big)\frac{\partial}{\partial v}\Bigg) \\ \nonumber &+ \frac{1}{\hbar A(w)}\Bigg(2pq(q-1)\frac{\partial}{\partial q} + \bigg(q(q-1)R'(q) - p - \frac{Q_0'(q)\kappa}{2p}-A(w)\frac{\partial p}{\partial w}\bigg)\frac{\partial}{\partial v} \Bigg),  
\end{align}
where $\kappa \coloneqq 2vq(q-1) + 2q - 1$.  \par 
To write down the metric we take the unique metric in the conformal class (up to a constant scale) defined by the twistor distribution that satisfies $\Omega_{-} = \frac{1}{2}dw \wedge da$ (this is the unique metric in the conformal class compatible with the period structure in the sense condition (1) of Definition \ref{J} is satisfied).  This yields the hyper-K\"ahler metric
\begin{align}
g &= da \Bigg( \frac{\kappa}{2pq(q-1)}dw - \frac{A(w)}{2pq(q-1)}dq\Bigg) - \bigg(\frac{\partial p}{\partial a}\bigg)^{-1} dw \,  dv \nonumber  \\ +& \bigg(\frac{\partial p}{\partial a}\bigg)^{-1}\Bigg( \frac{q(q-1)S'(q)-v}{A(w)} - \frac{\kappa\big(q(q-1)R'(q)-p-\frac{Q_0'(q)\kappa}{2p}-A(w)\frac{\partial p}{\partial w}\big)}{2pq(q-1)A(w)}\Bigg)dw\,  dw \nonumber \\ &+ \bigg(\frac{\partial p}{\partial a}\bigg)^{-1} \Bigg(\frac{\big(q(q-1)R'(q)-p-\frac{Q_0'(q)\kappa}{2p} - A(w)\frac{\partial p}{\partial w}\big)}{{2pq(q-1)}} \Bigg)dw \, dq .
\end{align}
A computation with the MAPLE \texttt{DifferentialGeometry} package confirms that the Ricci tensor vanishes as required.  \par 
We consider the flow generated by the vector field $V$ above.  
Set $\dot{w} = 1$ so that
\begin{align}
\frac{dq}{dw} = \frac{1}{A(w)}\Bigg(\frac{2pq(q-1)}{\hbar} + 2vq(q-1) + 2q -1 \Bigg).  
\end{align}
Differentiating again,  we obtain,  after a lengthy calculation: 
\begin{prop}
The flow generated by the vector field $V$ given by \textup{(\ref{PVIgenerator})} is the Painlev\'e VI equation
\begin{align}
\nonumber \frac{d^2q}{dw^2} = &\frac{1}{2}\bigg( \frac{1}{q} + \frac{1}{q-1}+\frac{1}{q-w}\bigg)\bigg(\frac{dq}{dw}\bigg)^2 - \bigg(\frac{1}{w} + \frac{1}{w-1} + \frac{1}{q-w}\bigg)\frac{dq}{dw} \\ &+\frac{1}{2}\frac{q(q-1)(q-w)}{w^2(w-1)^2}\bigg(1-\frac{w}{q^2}+\frac{w-1}{(q-1)^2}\bigg).  
\end{align}
\end{prop}
Using the notation of \cite{F} this is the Painlev\'e VI equation with fixed parameters $(k_\infty,  k_0,  k_1, k_t) = (1,1,1,1)$.  There are explicit algebraic solutions (see \cite {EGH}) defined by 
\begin{align}
(q^2-w)(q^2-2q+w)(q^2-2qw-w) = 0.  
\end{align}
The Euler vector field is given by 
\begin{align}
E = 2a\frac{\partial}{\partial a}
\end{align}
and satisfies $g(E,E) = 0$.  Such metrics are known (\cite{DM24},  Proposition 5.4) to take a form generalising that of Sparling and Tod \cite{ST} in any Pleba\'nski coordinate system.  The linearised data on $M$ of the Joyce metric $g^J$ and Joyce function $F$ (see \S 5 of \cite{B4} for definitions) vanish,  for this $g$.  The easiest way to see this is that the formula (73) in \cite{B4} for $F$ vanishes for the general form of the metric given in \cite{DM24}.   \par 
The metric $g$ admits a Killing vector
\begin{align}
K = \frac{aw(w-1)}{2pq(q-1)(q-w)}\frac{\partial}{\partial v}. 
\end{align}
This can be explained by noting that $K(Q_1(x)) = Q_0(x)$.  Only Joyce structures on moduli spaces consisting of quadratic differentials with simple poles may admit a vector field satisfying this last equation.  It would be interesting to test the conjecture that all such Joyce structures admit an analogous Killing vector.  

\section{Conclusion and outlook}
We have described a family of Joyce structures and provided a geometric description for their complex hyper-K\"ahler metrics: Each metric is described as the equivalent data of an $\mathcal{O}(2)$-valued $2$-form on twistor space naturally arising via the intersection pairings on the family of curves defined by $X$.  By construction,  in the coordinates $\{a_{i}^{(\alpha)}\hspace{-1mm},w_\alpha,  q_I,  v_I\}$ the components of the metric can be calculated via residue formulae which reduce to finite sums of Laurent coefficients.  Accordingly,  these hyper-K\"ahler metrics are amenable to explicit calculation,  especially with the assistance of computer algebra.  \par  
Quadratic differentials having poles some of which are of even order were not covered here but we expect the constructions in this work can be readily adapted.  The complication caused by poles of even order is that the tautological $1$-form $\psi$ on $\Sigma_0(p)$ (defined analogously to \S \ref{M_defn}) will have residues that are not constant on $\text{Quad}(\bar{m})$.  The conjectural Joyce structure of \cite{BdM} exists on a submanifold of $\text{Quad}(\bar{m})$ where all these residues vanish.  \par 

 Joyce structures are expected to exist on moduli spaces of meromorphic quadratic differentials on higher genus Riemann surfaces,  varying both the quadratic differential and the Riemann surface.  The methods here are insufficient but one may speculate that these methods can be adapted to obtain the structure on a submanifold corresponding to fixing the Riemann surface.  \par
 There are some aspects of the Joyce structures constructed here that warrant discussion.  The hyper-K\"ahler structures each admit a \textit{projectable hyper-Lagrangian} foliation: a rank-$2n$ distribution Lagrangian for the sphere of complex structures and pushing-down to a foliation of $M$,  Lagrangian for the intersection pairing $\omega$.  In a subsequent article we will explain that this foliation is related to the Hodge structure of the curves $\Sigma_0(p)$. \par
There is a natural procedure,  starting from a Joyce structure,  to construct a flat connection and compatible (potentially degenerate) symmetric bilinear form $g^J$ on $M$.  The existence of this data depends on the regularity of the Joyce structure in the limit as the period coordinates $\theta_a$ go to zero (a limit where the apparent singularities $q_I$ will approach branch points of $\Sigma_0(p)$).  The metric $g^J$,  if it exists,  can be calculated in terms of the second derivatives of a \textit{Joyce function} $F$.  In simple examples,  this procedure produces holomorphic data on $M$ \cite{BdM},  and even coincides with that of a known Frobenius structure on $M$ \cite{BM}.  We will provide a general formula for $F$,  of the Joyce structures constructed here,  in subsequent work.  
 \subsection*{Acknowledgements}
The author would like to thank their PhD supervisor Maciej Dunajski for guidance and patience,  Tom Bridgeland and Fabrizio Del Monte for helpful discussions,  and the referees for their careful reading and comments,  which have substantially improved the manuscript.  The author is supported by Cambridge Australia Scholarships and gratefully acknowledges the hospitality of the Simons Center for Geometry and Physics,  Stony Brook University at which some of the research for this paper was performed,  during the programme \textit{Geometric, Algebraic, and Physical structures around the moduli of Meromorphic Quadratic Differentials}. This work has been partially supported by STFC consolidated grant ST/X000664/1.\par 

\makeatletter
\def\@seccntformat#1{%
  \expandafter\ifx\csname c@#1\endcsname\c@section\else
  \csname the#1\endcsname\quad
  \fi}
\makeatother

\appendix
\section{A.  Twistor distributions and quaternionic structures}\label{q_integrability}
We prove the following standard fact (usually phrased in terms of the Obata connection \cite{S},  Section 6):
\begin{prop}
The distribution $L(\hbar)$ given by \textup{(\ref{Lax_distribution})} is integrable for all $\hbar$ if and only if the complex structures $I,J,K$ defined by \textup{(\ref{quaternions})} are integrable.  
\end{prop}
\begin{proof}
The forward implication is straightforward since we can identify the limits of $L(\hbar)$ as $\hbar \to 0$ and $\hbar \to \infty$ as the $\pm i$-eigenspaces of $I$,  while $L(\pm i)$ gives the $\pm i$-eigenspaces of $J$ and so on.  Conversely,  suppose $I$ and $J$ are integrable.  There are functions $\alpha_{ab}{}^{c},  \beta_{ab}{}^{c},  \gamma^\pm_{ab}{}^{c}$,  $a,b,c = 1,...,2n$ satisfying 
\begin{align}
[U_{a},U_{b}] &= \sum_{c=1}^{2n}\alpha_{ab}{}^{c}U_{c},  \quad 
[V_{a},V_{b}] = \sum_{c=1}^{2n}\beta_{ab}{}^{c}V_{c},\\ 
[U_{a} \pm iV_{a},  U_{b} \pm iV_{b}] 
&= \pm i([U_{a},V_{b}] + [V_{a},U_{b}])  + \sum_{c=1}^{2n}\alpha_{ab}{}^{c}U_{c} - \sum_{c=1}^{2n}\beta_{ab}{}^{c}V_{c} \nonumber \\
&= \sum_{c=1}^{2n} \gamma_{ab}^{\pm}{}^{c}(U_{c}\pm i V_{c}).  \label{J_eig}
\end{align}
Adding the $\pm$ equations from (\ref{J_eig}),  and using linear independence of $\{U_{a},V_{a}\}_{a=1}^{2n}$:
\begin{align}
\alpha_{ab}{}^{c} &= \frac{1}{2}(\gamma_{ab}^{+}{}^{c} +\gamma_{ab}^{-}{}^{c}),   \quad
\beta_{ab}{}^{c} = -\frac{i}{2}(\gamma_{ab}^{+}{}^{c} - \gamma_{ab}^{-}{}^{c}). 
\end{align}
Lastly,  we use these to calculate
\begin{align}
[U_{a} + \hbar^{-1}V_{a},  U_{b} + \hbar^{-1}V_{b}]  = \sum_{c=1}^{2n} \rho_{ab}{}^{c}(U_{c} + \hbar^{-1}V_{c}),  
\end{align}
where
\begin{align}
\rho_{ab}{}^{c} = (\alpha_{ab}{}^{c} + \hbar^{-1}\beta_{ab}{}^{c}).  
\end{align}
\end{proof}
\section{B.  Explicit Isomonodromic Flows} \label{explicit_flows}
We provide formulae for the generators of the isomonodromic flows of the ODE (\ref{schro}) with potential $Q(x)$ detailed in \S \ref{potential_setup},  in terms of parameters $\{a_i^{(\alpha)},w_\alpha,  v_I,  q_{I}\}$.  \par
There are $n$ linearly independent flows satisfying definition \ref{isopotential_defn},  we called \textit{isopotential flows}.  They are somewhat ``trivial" in that they preserve $Q(x)$.  They are given by,  for $1 \le i \le m_\alpha$ where $\alpha = 1,...,N$ and $0 \le i \le m_\infty - 4$ for $\alpha = \infty$
\begin{align}
L_{i}^{(\alpha)} \coloneqq U_{i}^{(\alpha)} + \frac{V_{i}^{(\alpha)}}{\hbar},  
\end{align}
where simply
\begin{align}
U_{i}^{(\alpha)}  \coloneqq \frac{\partial}{\partial a_i^{(\alpha)}},  
\end{align}
while the $\hbar^{-1}$ component is given by 
\begin{align}
V_{i}^{(\alpha)} \coloneqq - \sum_{I=1}^{n} \frac{\partial p_I}{\partial a_i^{(\alpha)}} \frac{\partial}{\partial v_I}.  
\end{align}
Next,  to complete the basis of $2n$ flows we take a flow satisfying (\ref{fuchs}) with $A = (x-q_I)^{-1}$ for $I = 1,...,n$.  We denote by $V_I$ the unique flow satisfying (\ref{fuchs}) for this choice of $A$ and $\dot{a_i}^{(\alpha)}=0$ for $1 \le i \le m_\alpha$ where $\alpha = 1,...,N$ and $0 \le i \le m_\infty - 4$ for $\alpha = \infty$.  Explicitly
\begin{align}
L_I = U_I + \frac{V_I}{\hbar}
\end{align}
where the $\hbar^{-1}$ component is
\begin{align}
V_I = &-2p_I\frac{\partial}{\partial q_I} - \sum_{\substack{K=1}}^n\Bigg( \sum_{i=0}^{2m_\infty-7}\frac{T_i^{(\infty)}(q_I)q_K^i}{2p_K} +  \sum_{\alpha=1}^{N}\sum_{i=1}^{2m_\alpha-1}\frac{T_i^{(\alpha)}(q_I-w_\alpha)}{2p_K(q_K-w_\alpha)^{i}} \nonumber \\ &+  \sum_{\alpha=1}^{N} \sum_{j=1}^{2m_\alpha-1} \frac{ja_{j}^{(\alpha)}}{2p_K(q_I-w_\alpha)(q_K-w_\alpha)^{j+1}} \Bigg)\frac{\partial}{\partial v_K} \\ &+ \sum_{\substack{K=1 \\ K \ne I}}^n \Bigg(\frac{Q_0'(q_K)}{2p_K(q_K-q_I)}-\frac{p_K}{(q_K-q_I)^2}\Bigg)\frac{\partial}{\partial v_K} \nonumber \\ &+ \Bigg( \frac{Q_0'(q_I)v_I}{p_I} - R'(q_I)+ \sum_{\substack{K=1 \\ K \ne I}}^n\frac{p_K}{(q_K-q_I)^2} \Bigg)\frac{\partial}{\partial v_I},
\end{align}
while the $\hbar^{0}$ component is 
\begin{align}
U_I = &\sum_{i = m_\infty - 3}^{2m_\infty-7}T_i^{(\infty)}(q_I)\frac{\partial}{\partial a_i^{(\infty)}} + \sum_{\alpha = 1}^{N} \sum_{i = m_\alpha+1}^{2m_\alpha - 1}T_i^{(\alpha)}(q_I-w_\alpha)\frac{\partial}{\partial {a_i^{(\alpha)}}} + \sum_{\alpha = 1}^{N} \frac{1}{q_I-w_\alpha}\frac{\partial}{\partial w_\alpha} \nonumber \\ 
&- 2v_I\frac{\partial}{\partial q_I} - \sum_{\substack{K = 1 \\  K \ne I}}^{n} \frac{1}{q_K-q_I}\frac{\partial}{\partial q_K} + \Bigg(\sum_{K \ne I}^{n} \frac{v_K}{(q_K-q_I)^2} - \sum_{K \ne I}^{n} \frac{3}{2(q_K-q_I)^{3}}  - S'(q_I)\Bigg)\frac{\partial}{\partial v_I} \nonumber \\
&+ \sum_{\substack{K = 1 \\ K \ne I}}^{n} \Bigg(\frac{3}{2(q_K-q_I)^{3}}- \frac{v_K}{(q_K-q_I)^{2}}  \Bigg)\frac{\partial}{\partial v_K}.  
\end{align}
Here 
\begin{align} 
T_i^{(\alpha)}(x) &= \sum_{k=i}^{2m_\alpha-1}(2i-k-2)\frac{a_k^{(\alpha)}}{x^{k-i+2}},  \\ 
 T_{i}^{(\infty)}(x) &= (2i-2m_\infty + 7)x^{2m_\infty-i-7} + \sum_{k=i+2}^{2m_\infty - 7}(2i-k+2)x^{k-i-2}a_k^{(\infty)}
 \end{align}
 and the functions $Q_0(x), R(x),S(x)$ are defined in \S \ref{potential_setup}.  
 
 \section{C.  Generic non-degeneracy of the curve $\Sigma(\xi,\hbar)$}\label{generic_non_singularity}
 In this appendix,  we show that the algebraic curve $\Sigma(\xi,\hbar)$,  defined by $y^2 = Q(x)$ is non-degenerate for generic $(\xi,\hbar) \in X \times \mathbb{C}^*$.  
 \begin{prop}
The rational function $Q(x)$ has simple zeroes given $(\xi,\hbar)$ lies in the complement of the vanishing set of a non-zero holomorphic function $\Delta: X \times \mathbb{C}^* \to \mathbb{C}$.  
\end{prop}
\begin{proof} 
The condition that $Q(x)$ has simple zeroes is equivalent to the non-vanishing of the discriminant $\Delta: X \times \mathbb{C}^* \to \mathbb{C}$ of the polynomial 
\begin{align}
\mathcal{P}(x) = \hbar^2\prod_{I=1}^{n}(x-q_I)^2\,  \prod_{\alpha=1}^{N}(x-w_\alpha)^{2m_\alpha -1} \, Q(x).
\end{align}
We therefore seek to show the discriminant does not vanish identically on $X \times \mathbb{C}^*$.  
We will show that for sufficiently small $\hbar \ne 0$,  there is a $(\xi,\hbar)$ such that $\mathcal{P}(x)$ has simple zeroes.  First,  note that
\begin{align}
\lim_{\hbar \to 0} \mathcal{P}(x) = \prod_{I=1}^{n}(x-q_I)^2 \,  \prod_{\alpha=1}^{N}(x-w_\alpha)^{2m_\alpha -1} \,  Q_0(x),  
\end{align}
which is a polynomial with $n - N - 1$ zeroes of multiplicity one by the assumption that $Q_0(x)$ has distinct zeroes and $n$ zeroes of multiplicity $2$ at $x = q_I$.  Since the position of distinct zeroes will (locally) vary holomorphically with $\hbar$,  it suffices to show that for small $\hbar \ne 0$,  $\mathcal{P}(x)$ has a pair of distinct zeroes in a neighbourhood of $x = q_I$.  We may set $v_I = 0$ for every $I$.  Define 
\begin{align}
\mathcal{F}_I(x) \coloneqq \prod_{\substack{I = 1 \\ I \ne J}}^n (x - q_J)^2 \prod_{\alpha=1}^N (x-w_\alpha) ^{2m_\alpha -1}.  
\end{align}
For each $I = 1,...,n$ we may expand
\begin{align}
\mathcal{P}(x) = \mathcal{F}_I(x)\Bigg( \frac{3}{4}\hbar^2 + \hbar p_I(x-q_I) + p_I^2(x-q_I)^2 + R_3(x) (x-q_I)^3\Bigg),  
\end{align}
where $R_3(x)$ is a rational function in $x$ regular at $q_I$.  \par Suppose that $\mathcal{P}(x)$ had a repeated zero for all $\hbar$ with $|\hbar| < \epsilon$.  Then,  for small enough $\epsilon$ and for some $I$ there exists a holomorphic function $a_I: B(0,\epsilon) \to \mathbb{C}$ of $\hbar$ with $a_I(0) = q_I$ and such that $\mathcal{P}(a_I(\hbar)) = 0$ and $\mathcal{P}'(a_I(\hbar)) = 0$.  Since we can take $\epsilon$ small enough so $\mathcal{F}_I(a_I(\epsilon)) \ne 0$,  we obtain from $\mathcal{P}(a_I(\hbar)) = 0$ and $\mathcal{P}'(a_I(\hbar)) = 0$ respectively:
\begin{align}
\frac{3}{4}\hbar^2 + \hbar p_I(a_I(\hbar) - q_I) + p_I^2(a_I(\hbar)-q_I)^2 + R_3(a_I(\hbar))  (a_I(\hbar)-q_I)^3 &= 0, \\ 
\hbar p_I + 2p_I^2(a_I(\hbar)-q_I) + 3R_3(a_I(\hbar)) (a_I(\hbar)-q_I)^2 + R_3'(a_I(\hbar)) (a_I(\hbar)-q_I)^3 &= 0.
\end{align}

Eliminating the $(a_I(\hbar)-q_I)^3$ and $(a_I(\hbar)-q_I)^0$ terms respectively yields the two equations
\begin{align}
\frac{9}{4} + 2p_I\frac{(a_I(\hbar)-q_I)}{\hbar} + p_I^2\frac{(a_I(\hbar)-q_I)^2}{\hbar^2} - R'_3(a_I(\hbar))\frac{(a_I(\hbar)-q_I)^4}{\hbar^2} = 0
\end{align}
and
\begin{align}
\frac{(a_I(\hbar)-q_I)}{\hbar}\Bigg(&\frac{1}{2}p_I + \frac{3R_3'(a_I(\hbar))(a_I(\hbar)-q_I)^2}{4p_I} + \frac{9R_3(a_I(\hbar))(a_I(\hbar)-q_I)}{4p_I} \nonumber \\  
& - p_I^2 \frac{(a_I(\hbar)-q_I)}{\hbar} - R_3(a_I(\hbar))\frac{(a_I(\hbar)-q_I)}{\hbar}(a_I(\hbar)-q_I)\Bigg) = 0.  
\end{align}
Now,  
\begin{align}
\lim_{\hbar \to 0} \frac{a_I(\hbar) - q_I}{\hbar} = a_I'(0)
\end{align}
so that taking the $\hbar \to 0$ limits yields a pair of quadratic equations in $ a_I'(0)$:
\begin{align}
\frac{9}{4} + 2p_Ia_I'(0) + p_I^2a_I'(0)^2  &= 0,  \\ 
a_I'(0)\bigg(\frac{1}{2}p_I  - p_I^2 a_I'(0) \bigg) &= 0.  
\end{align}
This system has no solution for $a_I'(0)$.  \end{proof}
 \section{D.  Linear independence of isomonodromic flows}\label{linear_independence_appendix}
 In this appendix we show that the set of $4n$ vector fields
 $$\mathcal{S} = \{U_I,V_I\} \cup \{U_i^{(\alpha)},V_i^{(\alpha)}\} $$ 
given in Appendix \ref{explicit_flows} are linearly independent,  hence trivialise $TX$.  This would then complete the proof of Proposition \ref{explicit_flows_calc_prop} by showing that the basis of isomonodromic flows $\{L_i^{(\alpha)}\} \cup \{L_I\}$ define a twistor distribution of rank $2n$ in the sense of Definition \ref{twistor_distribution_def}.  \par 
First,  note that the linear independence of the $n$ vector fields $\{V_{i}^{(\alpha)}\}$
    is equivalent to the invertibility of a particular $n \times n$ matrix $\mathcal{N}$ (related to a Vandermonde matrix)
with entries $\mathcal{N}^I_{(\alpha,i)}$,  letting $(\alpha, i)$ index the column and $I$ index the row:
\begin{align}
\mathcal{N}_{(\alpha,i)}^I \coloneqq  2p_I\frac{\partial p_I}{\partial a_{i}^{(\alpha)}} = \frac{1}{(q_I-w_\alpha)^i}
\end{align}
for $\alpha = 1,...,n$,  $i = 1,...,m_\alpha$,  $I = 1,...,n$ and 
\begin{align}
\mathcal{N}_{(\infty,i)}^I \coloneqq 2p_I\frac{\partial p_I}{\partial a_{i}^{(\infty)}} = q_I^i
\end{align}
for $i = 0,...,m_\infty - 4$,  $I = 1,...,n$.  \par 
We will see later $\operatorname{det}(\mathcal{N)} \ne 0$ on $X$.  
Given this non-vanishing,  we see that the $v_I$ coordinate derivatives for $I = 1,...,n$ are in the span of $\mathcal{S}$.  We can thus instead consider a set $\tilde{\mathcal{S}}$ with $\text{span}(\tilde{\mathcal{S}}) = \text{span}(\mathcal{S})$ given by 
\begin{align}
\tilde{S} = \{\tilde{U_I},\tilde{V_I}\} \cup \{U_i^{(\alpha)},V_i^{(\alpha)}\},  
\end{align}
where the $\tilde{V}_I,  \tilde{U}_I$ are $V_I,  U_I$ minus their components in the $v_I$ directions for $I = 1,...,n$.  The $\tilde{V_I}$ are then clearly linearly independent and have the same span as the $q_I$ coordinate derivatives on $X$ for $I = 1,...,n$.  In particular,  the vertical bundle is the span of $S$ and so it suffices to check that the projections of $\{U_{i}^{(\alpha)}\} \cup \{U_I\}$ to $M$ are linearly independent.  The $\{U_{i}^{(\alpha)}\}$ are clearly linear independent (being the coordinate vector fields) and transverse to the span of the $\{U_I\}$ so it remains to the check the projections of the $n$ vector fields $\{U_I\}$ are linearly independent.  \par
It is convenient to set $a_{2m_\alpha}^{(\alpha)} \coloneqq w_\alpha$.  
We seek to show that
\begin{align}
\Bigg\{ \sum_{i = m_\infty - 3}^{2m_\infty-7} \hspace{-2mm} T_i^{(\infty)}(q_I)\frac{\partial}{\partial a_i^{(\infty)}} + \sum_{\alpha = 1}^{n} \sum_{i = m_\alpha+1}^{2m_\alpha - 1} \hspace{-2mm} T_i^{(\alpha)}(q_I-w_\alpha)\frac{\partial}{\partial {a_i^{(\alpha)}}} + \sum_{\alpha = 1}^{N} \frac{1} {q_I-w_\alpha}\frac{\partial}{\partial a_{2m_\alpha}^{(\alpha)}} \Bigg\}_{I=1}^{n}
\end{align}
is a linearly independent set.  We argue that the matrix $\hat{\mathcal{N}}$ expressing the above as a linear combination of the linearly independent set
\begin{align}
\Bigg\{\frac{\partial}{\partial {a_i^{(\alpha)}}} \Bigg\}_{\alpha = 1,...,N}^{i = m_\alpha+1,...,2m_\alpha} \cup \Bigg\{   \frac{\partial}{\partial a_i^{(\infty)}}\Bigg\}^{i=m_\infty-3,...,2m_\infty-7} 
\end{align}
is invertible if and only if $\mathcal{N}$ is.  This follows because we may express $\hat{\mathcal{N}} = \mathcal{N}\mathcal{T}$, where $\mathcal{T}$ is block diagonal,  with $N+1$,  $m_\alpha \times m_\alpha$ blocks labelled by $\alpha = 1,...,N,\infty$ corresponding to the poles.  Each block is triangular with non-vanishing diagonal.  The entries of $\mathcal{T}$ comes from the coefficients of the functions $T_i^{(\alpha)}$ and $T_i^{(\infty)}$ written down in Appendix \ref{explicit_flows}.  Schematically,  for $\alpha = 1,...,N$ we have
\begin{align}
&\begin{bmatrix}
\frac{1}{q_1-w_\alpha} & \frac{1}{(q_1-w_\alpha)^2} & \cdots & \frac{1}{(q_1-w_\alpha)^{m_\alpha}}  \\
\frac{1}{q_2-w_\alpha} & \frac{1}{(q_2-w_\alpha)^2} & \cdots & \frac{1}{(q_2-w_\alpha)^{m_\alpha}}  \\ 
\vdots & \vdots & \vdots & \vdots \\
\frac{1}{q_n-w_\alpha} & \frac{1}{(q_n-w_\alpha)^2} & \cdots & \frac{1}{(q_n-w_\alpha)^{m_\alpha}}  
\end{bmatrix}
\hspace{-3mm}
\begin{bmatrix}
  1 & 0 & \cdots & 0 \\ 
0 & {(2m_\alpha-3)a_{2m_\alpha-1}^{(\alpha)}} & \cdots  & (m_\alpha-1)a_{m_\alpha+1}  \\ 
\vdots & \vdots & \vdots & \vdots \\ 
0 & 0 & \cdots & {a_{2m_\alpha-1}}
\end{bmatrix}  \nonumber \\
&= 
\begin{bmatrix}
\frac{1}{q_1-w_\alpha} & T^{(\alpha)}_{2m_\alpha-1}(q_1-w_\alpha) & \cdots & T^{(\alpha)}_{m_\alpha+1}(q_1-w_\alpha)\\
\frac{1}{q_2-w_\alpha} & T^{(\alpha)}_{2m_\alpha-1}(q_2-w_\alpha) & \cdots & T^{(\alpha)}_{m_\alpha+1}(q_2-w_\alpha) \\
\vdots & \vdots & \vdots & \vdots \\
\frac{1}{q_n-w_\alpha} & T^{(\alpha)}_{2m_\alpha-1}(q_n-w_\alpha) & \cdots & T^{(\alpha)}_{m_\alpha+1}(q_n-w_\alpha)\\
\end{bmatrix},  
\end{align}
where the second factor is an $m_\alpha \times m_\alpha$ block of $\mathcal{T}$.  The decomposition is similar for the block corresponding to $\alpha = \infty$.   \par 
So it remains to check that $\mathcal{N}$ is invertible on $X$.  Recall that on $X$ the $q_I$ are distinct and away from the branch points.  The poles $w_\alpha$ are of course also distinct.  Given these conditions,  we will show $\det(\mathcal{N})$ is non-vanishing.  \par 

We clear denominators and consider the determinant of the matrix
\begin{align}
\mathcal{M} = \prod_{\substack{\alpha = 1,...,N \\ I = 1,...,n}} (w_\alpha - q_I)^{m_\alpha} 
\mathcal{N}.  
 \end{align}
 The $(\alpha,  i)$ column has entries which are of degree $n \sum_{\beta=1}^N  m_\beta - i$,  for $\alpha \ne \infty$ and $n \sum_{\beta=1}^N  m_\beta + i$ for $\alpha = \infty$.  Accordingly,  considering $\det(\mathcal{M})$ as a homogeneous polynomial in $q_I,  w_\alpha$ we have: 
 \begin{align}
 &\phantom{=} \deg(\det(\mathcal{M})) = \sum_{\alpha = 1}^{N} \sum_{i=1}^{m_\alpha} \Big( n \sum_{\beta=1}^N  m_\beta - i\Big) + \sum_{i=0}^{m_\infty-4} \Big(n \sum_{\beta=1}^N  m_\beta + i \Big) \nonumber \\ 
 &=    \sum_{\alpha = 1}^{N} \bigg( nm_\alpha \sum_{\beta=1}^N  m_\beta - \frac{m_\alpha(m_\alpha+1)}{2}\bigg) + n(m_\infty -3) \sum_{\beta=1}^N  m_\beta + \frac{(m_\infty-3)(m_\infty-4)}{2} \nonumber  \\ \nonumber 
 &= n^2(n-m_\infty+3) -  \sum_{\alpha = 1}^{N}  \frac{m_\alpha(m_\alpha+1)}{2}  + \frac{(m_\infty-3)(m_\infty-4)}{2} \\
  &= n^2(n-m_\infty+3)  - \frac{n}{2}+ \frac{(m_\infty-3)^2}{2} -  \sum_{\alpha = 1}^{N}  \frac{m_\alpha^2}{2},   
 \end{align}
 where we have used the definition of $n$ in (\ref{ngendef}).  \par 
 We will now find all the linear factors of $\operatorname{det}(\mathcal{M})$:   \par 
 $(q_I - q_J)$ divides $\det(\mathcal{M})$ at least once for each pair $I \ne J$ since $q_I = q_J$ implies two columns of $\mathcal{M}$ are equal.  \par 
From thinking about the cofactor expansion we see that $\det(\mathcal{N})$ will have a pole of order $m_\alpha$ along $w_\alpha - q_I = 0$.  Accordingly,  $\det(\mathcal{M})$ will have a factor of $(w_\alpha - q_I)$ with multiplicity $(n-1)m_\alpha$.  \par 
We will show the multiplictiy of $(w_\alpha-w_\beta)$ is $m_\alpha m_\beta$.  First,  note that if it vanishes,  at least two columns of $\mathcal{N}$ are the same,  and so it is a factor.  Without loss of generality we take $m_\alpha \ge m_\beta$.  Make the substitution $w_\alpha = t + w_\beta$.  Then the columns $(\alpha,i)$ of the original matrix $\mathcal{N}$ for $i = 1,...,n$ are the only columns to depend on $t$.  We may bound the multiplicity below by repeatedly applying the ``cofactor product rule".  We have
\begin{align}
\frac{d}{dt}\det(\mathcal{N}) = \pm \hspace{-3mm} \sum_{i = 1,..,m_\alpha} (-1)^i \det(\mathcal{N}{[\alpha,i]}),  
\end{align}
where $\mathcal{N}{[\alpha,i]}$ is the matrix $\mathcal{N}$ with the $(\alpha,i)$ column replaced by its derivative with respect to $t$.  To calculate higher derivatives of $\det(\mathcal{N})$ we may expand the derivative of $\mathcal{N}{[\alpha,i]}$ using the same formula and so on.  \par 
Two columns of $\frac{d}{dt}\det(\mathcal{N})$ will be the same when $t = 0$ unless $m_\alpha = m_\beta = 1$.  In fact,  we must differentiate at least $m_\alpha m_\beta$ times until there exists a non-vanishing contribution to the expansion for the determinant when $t = 0$.  One way to see this is that we must decrease the degree of the $(i,\alpha)$ column with entries $(q_I - w_\beta - t)^i$ for $i \le m_\beta$ at least $m_\beta - i + 1$ times to have them distinct from the $(\beta,i)$ columns when $t = 0$.  This leaves $m_\beta$ columns with degree $-m_\beta - 1$ and $m_\alpha - m_\beta$ columns with degrees $-m_\beta - 1,...,-m_\alpha$.  To make all these columns distinct when $t = 0$ we need to differentiate at least another $m_\alpha - m_\beta + i - 1$ times for each $i = 1,...,m_\beta$.  Thus
\begin{align}
\frac{d^k}{dt^k}\det(\mathcal{N}) = 0
\end{align}
for $0 \le k < m_\alpha m_\beta$ and $(w_\alpha - w_\beta)$ has multiplicity at least $m_\alpha m_\beta$.   \par 
The number of linear factors accounted for above,  counting multiplicity,  is 
\begin{align}
&\sum_{I = 1}^{n} \sum_{\alpha = 1}^{N} (n-1)m_\alpha + \frac{n(n-1)}{2} + \sum_{\alpha > \beta}^{N} m_\alpha m_\beta\nonumber  \\ \nonumber = \ &n^3 - n(n-1)(m_\infty - 3)- \frac{n^2}{2} - \frac{n}{2} + \sum_{\alpha > \beta}^{N} m_\alpha m_\beta \\ = \ &n^3 - n(n-1)(m_\infty - 3)- \frac{n^2}{2} - \frac{n}{2} + \frac{1}{2} \Bigg( \sum_{\alpha = 1}^N m_\alpha \Bigg)\Bigg( \sum_{\beta = 1}^N m_\beta \Bigg) - \sum_{\alpha = 1}^{N} \frac{m_\alpha^2}{2}
\nonumber \\ = \ &n^3 - n(n-1)(m_\infty - 3)- \frac{n^2}{2} - \frac{n}{2} + \frac{1}{2}(n - m_\infty + 3)^2 -  \sum_{\alpha = 1}^{N}  \frac{m_\alpha^2}{2} \nonumber
\nonumber \\ = \ &n^3 - n^2(m_\infty - 3) - \frac{n}{2} + \frac{(m_\infty - 3)^2}{2} -  \sum_{\alpha = 1}^{N}  \frac{m_\alpha^2}{2} = \deg(\det(\mathcal{M})).  
\end{align} 
Accordingly,  all the factors of $\det(\mathcal{M})$ take the form $(q_I - q_J)$,  $(q_I - w_\alpha)$ or $(w_\alpha - w_\beta)$ and hence $\det(\mathcal{N})$ is non-vanishing on $X$.  

\subsection*{Data availability} No datasets were generated or analysed during this work.  
\subsection*{Conflicts of interest} The author declares there are no conﬂicts of interest.


\begin{thebibliography}{11}\raggedright

\bibitem{ACGH} E.  Arbarello,  M.  Cornalba,  P.A.  Griffiths,  J.  Harris,  
{\em Geometry of Algebraic Curves},  
Grundlehren der mathematischen Wissenschaften,  \textbf{267},  Springer,  (1985).  

\bibitem{BE} T N.  Bailey,  M.G.  Eastwood,  
{\em Complex Paraconformal Manifolds - their Differential Geometry and Twistor Theory},
Forum.  Math. \textbf{3}, 1,  (1991). 

\bibitem{BK} M.  Bertola,  D.  Korotkin,  
{\em Tau-Functions and Monodromy Symplectomorphisms},  
 Commun.  Math.  Phys.  \textbf{388},  (2021).

\bibitem{Bo} P. Boalch,  
{\em Symplectic Manifolds and Isomonodromic Deformations},
Adv.  Math. {\bf 163},  2,  (2001).

\bibitem{Bo2} A.  Bobenko,  
{\em Introduction to Compact Riemann Surfaces}, 
Computational Approach to Riemann Surfaces,  Lecture Notes in Mathematics,  Springer,  (2013).  
 
\bibitem{B2} T.  Bridgeland,  
{\em Joyce structures on spaces of quadratic differentials},
Geometry \& Topology \textbf{29},  5, (2025).

\bibitem{B4} T.Bridgeland,  
{\em Joyce structures and their twistor spaces},  
Adv.  Math.  {\bf 462},  110089, (2025).  

\bibitem{B1} T.  Bridgeland,
{\em Tau functions from Joyce structures},
SIGMA \textbf{20},  112, (2024). 

\bibitem{B3} T.  Bridgeland, 
{\em Geometry from Donaldson-Thomas invariants},
 Integrability, Quantization, and Geometry
II.  Quantum Theories and Algebraic Geometry,  Proc.  Sympos.  Pure Math.  Amer.  Math.  Soc., (2021).

\bibitem{BdM} T.  Bridgeland,  F.  Del Monte,  
 {\em Joyce Structures and Poles of Painlev\'{e} Equations}
 \texttt{arXiv:2505.03429v1},  (2025).

\bibitem{BM} T.  Bridgeland,   D.  Masoero,
{\em On the monodromy of the deformed cubic oscillator},
Math. Ann.  {\bf 385},  (2023).

\bibitem{BS2} T.  Bridgeland,  I. Smith,
{\em Quadratic Differentials as Stability Conditions},
 Publ.  math.  IHES {\bf 121},  (2015). 
 
\bibitem{BS} T.  Bridgeland,  I.A.B.  Strachan,
{\em Complex hyperkähler structures defined by Donaldson–Thomas invariants},
Lett.  Math.  Phys.  {\bf 111},  54,  (2021).

\bibitem{CS} A. \v{C}ap,  J. Slov\'ak, 
{\em Parabolic Geometries I: Background and General Theory},
Math.  Surv.  and Monographs {\bf 154},
Amer.  Math.  Soc.,  (2009).

\bibitem{D} M.  Dunajski,  
{\em Null Kähler Geometry and Isomonodromic Deformations},  
Commun.  Math.  Phys.  {\bf 391},  (2022).

\bibitem{DM24} M.  Dunajski,  T.  Moy,  
{\em Heavenly metrics,  hyper-Lagrangians and Joyce structures},
J.  Lond.  Math.  Soc.  {\bf 110},  5,  (2024).   

\bibitem{EGH} A.  Eremenko,  A. Gabrielov,  A.  Hinkkanen,  
{\em Exceptional solutions to the Painlev\'e VI equation},  
J.  Math.  Phys.  {\bf 58},  1, (2017).  

\bibitem{F} R.  Fuchs,  {\em Sur quelques \'{e}quations diff\'{e}rentielles lin\'{e}aires du second ordre}, Comptes Rendus de l'Acad\'{e}mie des Sciences Paris {\bf 141},  (1905). 

\bibitem{HKLR} N.J.  Hitchin,  A.  Karlhede, U. Lindström, M.  Ro\v{c}ek,  
{\em Hyperkähler metrics and supersymmetry},  
 Commun.Math. Phys.  \textbf{108},  (1987).

\bibitem{JMU} M.  Jimbo,  T.  Miwa,  K.  Ueno,  
{\em Monodromy preserving deformation of linear ordinary differential equations with rational coefficients},  Physica D \textbf{2},  2,  (1981).

\bibitem{J} D.  Joyce,  
{\em Holomorphic generating functions for invaraints counting coherent sheaves on Calabi-Yau 3-folds}, 
 Geom.  Topol.  \textbf{11},  (2007).  

\bibitem{JS} D.  Joyce,  Y.  Song,  
{\em A theory of generalized Donaldson-Thomas invariants},  
Mem.  Amer.  Math.  Soc.  \textbf{217},  no.  1020,  (2012).  

\bibitem{M} G.  Mahoux,  
{\em Introduction to the Theory of Isomonodromic Deformations of Linear Ordinary Differential Equations with Rational Coefficients},
The Painlevé Property,  One Century Later,  CRM Series in Mathematical Physics,  Springer,  (1999).  

\bibitem{Mu} D.  Mumford,  
{\em Tata Lectures on Theta II,  Jacobian theta functions and differential equations},  
Modern Birkhäuser Classics,  (2006).

\bibitem{O} K. Okamoto,  
{\em Studies on the Painlev\'{e} equations.  I: Sixth Painlev\'{e} equation PVI},
Ann.  Mat.  Pura Appl.  \textbf{146}(1),  1,  (1986). 

\bibitem{Pe} R.  Penrose,
{\em Nonlinear Gravitons and Curved Twistor Theory},
Gen.  Relativ.  Gravit.  {\bf 7},  1,  (1976).

\bibitem{Po} H.  Poincar\'{e}, 
{\em Sur les groupes des \'{e}quations lin\'{e}aires},
Acta Math.  {\bf 4},  (1884). 

\bibitem{S} S. M.  Salamon,  
{\em Differential geometry of quaternionic manifolds}, 
Ann.  Sci.  Éc.  Norm.  Supér.  Serie 4,  \textbf{19},  1,  (1986).  

\bibitem{ST} G. A. J.  Sparling,  K. P. Tod,
{\em An example of an H‐space},
 J.  Math.  Phys.  {\bf 22},  (1981). 

\bibitem{U} K.  Ueno,
{\em Monodromy Preserving Deformation of Linear Differential Equations with Irregular Singular Points},
Proc.  Japan.  Acad.  \textbf{56},  Ser. A,  (1980).  

\bibitem{Y} D.  Yamakawa,
{\em Fundamental two-forms for isomonodromic deformations}, 
Journal of Integrable Systems,  \textbf{4},  1, (2019). 

%
%

%
%
%
%
%
%
%
%
%
%
%
%
%
%
%
%
%
%
%
%
%
%
%
%
%
%
%
%
%
%

%
%
%
%
%
%
%
%

%
%
%
%
%
%
%
%
%
%
%
%
%
%
%
%
%
%
%
%
%
%
%
%
%
%
%
%
%
%
%
%
%
%
%
%
%
%
%
%
%
%
%
%
%
%
%


\end{thebibliography}
\end{document}